\documentclass[a4paper,reqno]{amsart}

\usepackage[british]{babel}

\numberwithin{equation}{section}

\usepackage[left=3.5cm,right=3.5cm,top=3.5cm,bottom=3.5cm,footskip=1.5cm,headsep=1cm]{geometry}
\usepackage[%
pdfauthor={Habib Ammari, Silvio Barandun, Jinghao Cao, Bryn Davies,
and Erik Orvehed Hiltunen},%
pdftitle={Mathematical foundations of the skin effect},
pdfsubject={Mathematical foundations of the skin effect},
pdfkeywords={Non-Hermitian systems, skin effect, subwavelength resonators, imaginary gauge
potential, generalised Brillouin zone, exceptional points, topological invariant, vorticity,
phase transition, bulk boundary correspondence},
]{hyperref}
\usepackage[T1]{fontenc} 
\usepackage{lmodern} 
\rmfamily 
\DeclareFontShape{T1}{lmr}{b}{sc}{<->ssub*cmr/bx/sc}{}
\DeclareFontShape{T1}{lmr}{bx}{sc}{<->ssub*cmr/bx/sc}{}
\usepackage{lipsum}
\usepackage{mathtools}
\usepackage{amsmath}
\usepackage{amssymb}
\usepackage{tikz}
\usetikzlibrary{matrix,positioning,decorations.pathreplacing,calc}
\usetikzlibrary{
decorations.text,%
decorations.markings,%
shadows}
\usepackage{adjustbox}
\usepackage{graphicx} 

\usepackage{caption}
\usepackage{subcaption}

\usepackage{dsfont} 
\usepackage[format=hang]{caption}

\usepackage[normalem]{ulem}

\newcommand{\abs}[1]{\lvert#1\rvert}
\usepackage{bm}

\usepackage[
style=numeric,
sorting=nty,
maxnames=99,
maxalphanames=5,
natbib=true,
sortcites]{biblatex}

\DeclareNameAlias{default}{family-given}

\graphicspath{{figures/}}

\AtEveryBibitem{
 \clearfield{url}
 \clearfield{issn}
 \clearfield{isbn}
 \clearfield{urldate}
 
 \ifentrytype{book}{
  \clearfield{pages}}{%
 }
}

\usepackage{xargs}                      
\usepackage[prependcaption,textsize=tiny,textwidth=3cm]{todonotes}
\newcommandx{\unsure}[2][1=]{\todo[linecolor=red,backgroundcolor=red!25,bordercolor=red,#1]{#2}}
\newcommandx{\change}[2][1=]{\todo[linecolor=blue,backgroundcolor=blue!25,bordercolor=blue,#1]{#2}}
\newcommandx{\info}[2][1=]{\todo[linecolor=OliveGreen,backgroundcolor=OliveGreen!25,bordercolor=OliveGreen,#1]{#2}}
\newcommandx{\improvement}[2][1=]{\todo[linecolor=black,backgroundcolor=black!25,bordercolor=black,#1]{#2}}
\newcommandx{\thiswillnotshow}[2][1=]{\todo[disable,#1]{#2}}
\usepackage{amsthm}
\usepackage[noabbrev,capitalize]{cleveref}
\crefname{proposition}{Proposition}{Propositions}
\crefname{equation}{}{}

\newtheorem{theorem}{Theorem}[section]
\newtheorem{lemma}[theorem]{Lemma}
\newtheorem{proposition}[theorem]{Proposition}
\newtheorem{corollary}[theorem]{Corollary}

\theoremstyle{definition}
\newtheorem{definition}[theorem]{Definition}

\crefname{assumption}{Assumption}{Assumptions}
\crefname{definition}{Definition}{Definitions}
\crefname{corollary}{Corollary}{Corollaries}
\crefname{enumi}{item}{items}

\usepackage{fancyhdr}

\fancypagestyle{plain}{%
\fancyhead[C]{}
\fancyfoot[C]{}

}
\fancypagestyle{nosection}{%
\fancyhead[CE]{}
\fancyhead[CO]{}
\fancyfoot[CE,CO]{\thepage}
}

\DeclareMathOperator{\N}{\mathbb{N}}
\DeclareMathOperator{\Z}{\mathbb{Z}}

\DeclareMathOperator{\R}{\mathbb{R}}
\DeclareMathOperator{\C}{\mathbb{C}}

\DeclareMathOperator{\tr}{tr}
\renewcommand{\i}{\mathbf{i}}
\renewcommand{\tilde}{\widetilde}

\renewcommand{\bar}[1]{\overline{#1}}
\newcommand{\inv}{^{-1}}

\renewcommand{\qedsymbol}{$\blacksquare$}

\usepackage{fancyhdr}

\fancypagestyle{plain}{%
\fancyhead[C]{}
\fancyfoot[C]{}

}
\fancypagestyle{nosection}{%
\fancyhead[CE]{}
\fancyhead[CO]{}
\fancyfoot[CE,CO]{\thepage}
}

\DeclareMathOperator{\dtn}{\mathcal{T}}
\DeclareMathOperator{\diag}{diag}
\DeclareMathOperator{\BO}{\mathcal{O}}
\DeclareMathOperator{\exactcapmat}{\mathbf{C}}
\DeclareMathOperator{\capmat}{\mathcal{C}}

\DeclareMathOperator{\capmatg}{\mathcal{C}^\gamma}
\DeclareMathOperator{\qpcapmatg}{\mathcal{C}^{\gamma,\alpha}}

\newcommand{\pri}{^\prime}
\newcommand{\prii}{^{\prime\prime}}
\renewcommand{\epsilon}{\varepsilon}
\DeclareMathOperator{\dd}{d\!}

\renewcommand{\i}{\mathbf{i}}
\renewcommand{\tilde}{\widetilde}

\DeclareMathOperator{\Ind}{Ind}
\renewcommand{\sp}{\mathrm{sp}}

\DeclareMathOperator{\iL}{{\mathsf{L}}}
\DeclareMathOperator{\iR}{{\mathsf{R}}}
\DeclareMathOperator{\iLR}{{\mathsf{L},\mathsf{R}}}

\usepackage{tcolorbox}
\usetikzlibrary{tikzmark,calc}

\DeclareMathOperator{\csch}{csch}

\DeclareMathOperator*{\argmin}{argmin}

\usepackage{enumerate}
\usepackage{upgreek}

\begin{document}

\title[Mathematical foundations of the skin effect]{Mathematical foundations of the non-Hermitian skin effect}

 \author[H. Ammari]{Habib Ammari}
\address{\parbox{\linewidth}{Habib Ammari\\
 ETH Z\"urich, Department of Mathematics, Rämistrasse 101, 8092 Z\"urich, Switzerland}}
\email{habib.ammari@math.ethz.ch}
\thanks{}

\author[S. Barandun]{Silvio Barandun}
 \address{\parbox{\linewidth}{Silvio Barandun\\
 ETH Z\"urich, Department of Mathematics, Rämistrasse 101, 8092 Z\"urich, Switzerland}}
 \email{silvio.barandun@sam.math.ethz.ch}

\author[J. Cao]{Jinghao Cao}
 \address{\parbox{\linewidth}{Jinghao Cao\\
 ETH Z\"urich, Department of Mathematics, Rämistrasse 101, 8092 Z\"urich, Switzerland}}
\email{jinghao.cao@sam.math.ethz.ch}

\author[B. Davies]{Bryn Davies}
 \address{\parbox{\linewidth}{Bryn Davies\\
Department of Mathematics, Imperial College London, 180 Queen's Gate, London SW7~2AZ, UK}}
\email{bryn.davies@imperial.ac.uk}

 \author[E.O. Hiltunen]{Erik Orvehed Hiltunen}
\address{\parbox{\linewidth}{Erik Orvehed Hiltunen\\
Department of Mathematics, Yale University, 10 Hillhouse Ave,
New Haven, CT~06511, USA}}
\email{erik.hiltunen@yale.edu}

\maketitle

\begin{abstract}
We study the skin effect in a one-dimensional system of finitely many subwavelength resonators with a non-Hermitian imaginary gauge potential. Using Toeplitz matrix theory, we prove the condensation of bulk eigenmodes at one of the edges of the system. By introducing a generalised (complex) Brillouin zone, we can compute spectral bands of the associated infinitely periodic structure and prove that this is the limit of the spectra of the finite structures with arbitrarily large size. Finally, we contrast the non-Hermitian systems with imaginary gauge potentials considered here with systems where the non-Hermiticity arises due to complex material parameters, showing that the two systems are fundamentally distinct. 
\end{abstract}

\date{}

\bigskip

\par

\bigskip

\\

\vfill
\hrule
\tableofcontents
\vspace{-0.5cm}
\hrule
\vfill

\section{Introduction} 

Understanding the skin effect in non-Hermitian physical systems has been one of the hottest research topics in recent years \cite{zhang.zhang.ea2022review, okuma.kawabata.ea2020Topological,lin.tai.ea2023Topological,yokomizo.yoda.ea2022NonHermitian,skinadd4,skinadd5}. The skin effect is the phenomenon whereby the bulk eigenmodes of a non-Hermitian system are all localised at one edge of an open chain of resonators. This phenomenon is unique to non-Hermitian systems with non-reciprocal coupling. It has been realised experimentally in topological photonics, phononics, and other condensed matter systems \cite{ghatak.brandenbourger.ea2020Observation,skinadd1,skinadd2, skinadd3}. 
By introducing an imaginary gauge
potential, all the bulk eigenmodes condensate in the same direction \cite{hatano,yokomizo.yoda.ea2022NonHermitian,rivero.feng.ea2022Imaginary}. This has deep implications for the fundamental physics of the system. For example, the non-Hermitian skin effect means that the conventional bulk-edge correspondence principle is violated. 

In this paper, we study the non-Hermitian skin effect in a deep subwavelength regime using first-principles mathematical analysis. The ultimate goal of subwavelength wave physics is to manipulate waves at subwavelength scales. Recent breakthroughs, such as the emergence of the field of metamaterials, have allowed us to do this in a way that is robust and that beats traditional diffraction limits. Here, we consider one-dimensional systems of high-contrast subwavelength resonators as a demonstrative setting to develop a mathematical and numerical framework for the non-Hermitian skin effect in the subwavelength regime.

We consider finite chains of subwavelength resonators with an imaginary gauge potential supported inside the resonators. This imaginary gauge potential is realised by adding an additional first-order term to the classical Helmholtz differential problem (the $\gamma$-term in \eqref{eq:coupled ods}). It is worth emphasizing that the $\gamma$-term is added only inside the subwavelength resonators, which are much smaller than the operating wavelength. Without  exciting the structure's subwavelength resonances, the effect of the $\gamma$-term would be negligible. Using an asymptotic methodology that was previously developed for Hermitian systems \cite{ammari.davies.ea2021Functional,feppon.cheng.ea2023Subwavelength}, 
we can approximate the subwavelength eigenfrequencies and eigenmodes of a finite chain of resonators by the eigenvalues of a matrix, which we call the \emph{gauge capacitance matrix}. When $\gamma=0$, this capacitance matrix reduces to the capacitance matrix previously derived for Hermitian systems of subwavelength resonators \cite{feppon.cheng.ea2023Subwavelength,ammari.barandun.ea2023Edge}. This traditional capacitance matrix was originally developed to model the analogous many-body problem in the setting of electrostatics \cite{maxwell1873treatise, diaz2011positivity}. Given this matrix formulation, we can use the well-established theory of tridiagonal Toeplitz matrices and operators to prove the condensation of eigenmodes at one of the edges of the finite chain. 

Using established spectral theory for tridiagonal Toeplitz matrices, we will show the topological origins of the non-Hermitian skin effect. That is, we will demonstrate the  condensation of the eigenmodes. As a consequence of the condensation of the eigenmodes, the imaginary gauge potential leads to a mismatch between the eigenvalues of a finite chain and the band structure of the associated infinite periodic structure, when the band structure is computed over the standard Brillouin zone (as used for Hermitian problems). However, the correspondence between the spectra of finite and infinite structures can be understood if the spectral bands of the infinite periodic structure are computed over the complex plane, in the sense that the quasiperiodicities must be allowed to be complex \cite{borisov.fedotov2022Bloch}. The eigenmodes behave as Bloch waves with a complex wave number, with the imaginary part describing the exponential spatial decay of the amplitude. We will show in Theorem~\ref{cvthm} that the generalised Brillouin zone introduced in (\ref{gbz}) gives the correct description of the limiting problem.

Although the spectrum of the infinite periodic problem computed over the standard Brillouin zone does not provide the correct approximation of the spectrum of a finite resonator chain, the problem is nonetheless insightful to investigate in its own right. We will  consider exceptional points where eigenvalues and eigenvectors coincide and the quasiperiodic gauge capacitance matrix is not diagonalisable  \cite{ammari.davies.ea2022Exceptional}
and prove their existence for periodic dimer structures at the edges of the Brillouin zone. We will also show that the set of exceptional points in terms of the separating distance between the two identical resonators (which constitutes the dimer) in the unit cell separates two different phases described by different values of  the winding number of the eigenvalues, which is known as the \emph{vorticity}.  
  
The imaginary gauge potential considered here is not the only way to break the Hermiticity of a system of subwavelength resonators. This is also achieved by making the material parameters complex valued, corresponding to sources of energy gain or loss in the system. An important class are PT-symmetric systems, which satisfy a combination of time-reversal and inversion symmetry. They inspired a plethora of explorations in topological photonics and phononics and their mathematical foundations \cite{ammari.davies.ea2022Exceptional,ptsymmetry1,ptsymmetry2,ptsymmetry3}. Based on the contrast with the previous work \cite{ammari.davies.ea2022Exceptional, ammari.hiltunen2020Edge}, it is natural to divide non-Hermitian systems into two classes depending on whether they can be reduced to Hermitian systems through an imaginary gauge transformation. If such reduction is possible (which we examine in \cref{sec: PT}), the eigenmodes will be closely approximated by Bloch modes over the standard Brillouin zone, and the vorticity will vanish. This means,
on the one hand, that vorticity can be considered as a topological indicator of the appearance of the skin effect and, secondly, that the spectral convergence of large structures with vanishing vorticity can be understood using the standard approaches for Hermitian systems, based on computations over the standard Brillouin zone \cite{ammari.davies.ea2023Spectral}. Conversely, in the setting considered in this work, this convergence does not hold as the system is fundamentally distinct from any Hermitian system. The current work is the first to establish the convergence theory when such reduction is not possible and the eigenmodes are described through the generalised (complex) Brillouin zone.

The paper is organised as follows. In \cref{sec: preliminaries} we present the mathematical setup of the problem and derive an asymptotic approximation of the subwavelength eigenfrequencies and their associated eigenmodes. 
These are accurately approximated by the spectrum of the gauge capacitance matrix. \cref{sec: skin effect} is dedicated to the mathematical analysis of the skin effect. The condensation of the eigenmodes is proved and the non-uniform distribution of the subwavelength eigenfrequencies is illustrated. 
 \cref{sec: periodic case} studies the infinite periodic case. 
 In \cref{sec: periodic case1} we first take 
the quasiperiodicities to be real. We show that using the standard Brillouin zone does not lead to the correct modeling of the limit of the set of subwavelength eigenfrequencies as the size of the system goes to infinity. 
Nevertheless, the band functions over the standard Brillouin zone are used to define the vorticity, which we show that it is non-trivial in the current setting.
Such topological invariant can be used as an indicator of the skin effect since it is trivial for other non-Hermitian structures obtained by taking the material parameters inside the subwavelength resonators complex. In \cref{sec: complex bands}, we introduce the generalised (complex) Brillouin zone and prove the convergence of  the subwavelength eigenfrequencies of a finite system to the bands of the corresponding infinitely periodic structure computed over the generalised Brillouin zone. 
 In \cref{sec: teoplitz theory} we summarise the well-established theory of Toeplitz matrices and operators, while in \cref{sec: dimers} we present some additional numerical results for finite dimer systems with imaginary gauge potentials. In \cref{sec: approx band} we show how to obtain a generalised discrete Brillouin zone for different dimer systems. 
Finally, in \cref{sec: PT} we show that systems with non-Hermiticity arising from complex material parameters can be described in the limit when their sizes go to infinity by bands of corresponding infinitely periodic structures computed over the standard Brillouin zone.

\section{Preliminaries}\label{sec: preliminaries}       %

We consider a one-dimensional chain of $N$ disjoint subwavelength resonators $D_i\coloneqq (x_i^{\iL},x_i^{\iR})$, where $(x_i^{\iLR})_{1\<i\<N} \subset \R$ are the $2N$ extremities satisfying $x_i^{\iL} < x_i^{\iR} <  x_{i+1}^{\iL}$ for any $1\leq i \leq N$. We fix the coordinates such that $x_1^{\iL}=0$. We also denote by  $\ell_i = x_i^{\iR} - x_i^{\iL}$ the length of the $i$-th resonators,  and by $s_i= x_{i+1}^{\iL} -x_i^{\iR}$ the spacing between the $i$-th and $(i+1)$-th resonator. The system is illustrated in \cref{fig:setting}. We will use 
\begin{align*}
   D\coloneqq \bigcup_{i=1}^N(x_i^{\iL},x_i^{\iR})
\end{align*}
to symbolise the set of subwavelength resonators.

\begin{figure}[htb]
    \centering
    \begin{adjustbox}{width=\textwidth}
    \begin{tikzpicture}
        \coordinate (x1l) at (1,0);
        \path (x1l) +(1,0) coordinate (x1r);
        \path (x1r) +(0.75,0.7) coordinate (s1);
        \path (x1r) +(1.5,0) coordinate (x2l);
        \path (x2l) +(1,0) coordinate (x2r);
        \path (x2r) +(0.5,0.7) coordinate (s2);
        \path (x2r) +(1,0) coordinate (x3l);
        \path (x3l) +(1,0) coordinate (x3r);
        \path (x3r) +(1,0.7) coordinate (s3);
        \path (x3r) +(2,0) coordinate (x4l);
        \path (x4l) +(1,0) coordinate (x4r);
        \path (x4r) +(0.4,0.7) coordinate (s4);
        \path (x4r) +(1,0) coordinate (dots);
        \path (dots) +(1,0) coordinate (x5l);
        \path (x5l) +(1,0) coordinate (x5r);
        \path (x5r) +(1.75,0) coordinate (x6l);
        \path (x5r) +(0.875,0.7) coordinate (s5);
        \path (x6l) +(1,0) coordinate (x6r);
        \path (x6r) +(1.25,0) coordinate (x7l);
        \path (x6r) +(0.525,0.7) coordinate (s6);
        \path (x7l) +(1,0) coordinate (x7r);
        \path (x7r) +(1.5,0) coordinate (x8l);
        \path (x7r) +(0.75,0.7) coordinate (s7);
        \path (x8l) +(1,0) coordinate (x8r);
        \draw[ultra thick] (x1l) -- (x1r);
        \node[anchor=north] (label1) at (x1l) {$x_1^{\iL}$};
        \node[anchor=north] (label1) at (x1r) {$x_1^{\iR}$};
        \node[anchor=south] (label1) at ($(x1l)!0.5!(x1r)$) {$\ell_1$};
        \draw[dotted,|-|] ($(x1r)+(0,0.25)$) -- ($(x2l)+(0,0.25)$);
        \draw[ultra thick] (x2l) -- (x2r);
        \node[anchor=north] (label1) at (x2l) {$x_2^{\iL}$};
        \node[anchor=north] (label1) at (x2r) {$x_2^{\iR}$};
        \node[anchor=south] (label1) at ($(x2l)!0.5!(x2r)$) {$\ell_2$};
        \draw[dotted,|-|] ($(x2r)+(0,0.25)$) -- ($(x3l)+(0,0.25)$);
        \draw[ultra thick] (x3l) -- (x3r);
        \node[anchor=north] (label1) at (x3l) {$x_3^{\iL}$};
        \node[anchor=north] (label1) at (x3r) {$x_3^{\iR}$};
        \node[anchor=south] (label1) at ($(x3l)!0.5!(x3r)$) {$\ell_3$};
        \draw[dotted,|-|] ($(x3r)+(0,0.25)$) -- ($(x4l)+(0,0.25)$);
        \node (dots) at (dots) {\dots};
        \draw[ultra thick] (x4l) -- (x4r);
        \node[anchor=north] (label1) at (x4l) {$x_4^{\iL}$};
        \node[anchor=north] (label1) at (x4r) {$x_4^{\iR}$};
        \node[anchor=south] (label1) at ($(x4l)!0.5!(x4r)$) {$\ell_4$};
        \draw[dotted,|-|] ($(x4r)+(0,0.25)$) -- ($(dots)+(-.25,0.25)$);
        \draw[ultra thick] (x5l) -- (x5r);
        \node[anchor=north] (label1) at (x5l) {$x_{N-3}^{\iL}$};
        \node[anchor=north] (label1) at (x5r) {$x_{N-3}^{\iR}$};
        \node[anchor=south] (label1) at ($(x5l)!0.5!(x5r)$) {$\ell_{N-3}$};
        \draw[dotted,|-|] ($(x5r)+(0,0.25)$) -- ($(x6l)+(0,0.25)$);
        \draw[ultra thick] (x6l) -- (x6r);
        \node[anchor=north] (label1) at (x6l) {$x_{N-2}^{\iL}$};
        \node[anchor=north] (label1) at (x6r) {$x_{N-2}^{\iR}$};
        \node[anchor=south] (label1) at ($(x6l)!0.5!(x6r)$) {$\ell_{N-2}$};
        \draw[dotted,|-|] ($(x6r)+(0,0.25)$) -- ($(x7l)+(0,0.25)$);
        \draw[ultra thick] (x7l) -- (x7r);
        \node[anchor=north] (label1) at (x7l) {$x_{N-1}^{\iL}$};
        \node[anchor=north] (label1) at (x7r) {$x_{N-1}^{\iR}$};
        \node[anchor=south] (label1) at ($(x7l)!0.5!(x7r)$) {$\ell_{N-1}$};
        \draw[dotted,|-|] ($(x7r)+(0,0.25)$) -- ($(x8l)+(0,0.25)$);
        \draw[ultra thick] (x8l) -- (x8r);
        \node[anchor=north] (label1) at (x8l) {$x_{N}^{\iL}$};
        \node[anchor=north] (label1) at (x8r) {$x_{N}^{\iR}$};
        \node[anchor=south] (label1) at ($(x8l)!0.5!(x8r)$) {$\ell_N$};
        \node[anchor=north] (label1) at (s1) {$s_1$};
        \node[anchor=north] (label1) at (s2) {$s_2$};
        \node[anchor=north] (label1) at (s3) {$s_3$};
        \node[anchor=north] (label1) at (s4) {$s_4$};
        \node[anchor=north] (label1) at (s5) {$s_{N-2}$};
        \node[anchor=north] (label1) at (s6) {$s_{N-1}$};
        \node[anchor=north] (label1) at (s7) {$s_N$};
    \end{tikzpicture}
    \end{adjustbox}
    \caption{A chain of $N$ subwavelength resonators, with lengths
    $(\ell_i)_{1\leq i\leq N}$ and spacings $(s_{i})_{1\leq i\leq N-1}$.}
    \label{fig:setting}
\end{figure}
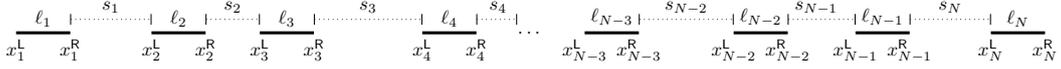

As a scalar wave field $u(t,x)$ propagates in a heterogeneous medium, it satisfies the following one-dimensional wave equation:
\begin{align*}
    \frac{1}{\kappa(x)}\frac{\partial{^2}}{\partial t^{2}}u(t,x) -\frac{\partial}{\partial x}\left(
        \frac{1}{\rho(x)}\frac{\partial}{\partial x}  u(t,x)\right) = 0, \qquad (t,x)\in\R\times\R.
\end{align*}
The parameters $\kappa(x)$ and $\rho(x)$ are the material parameters of the medium. We consider piecewise constant material parameters
\begin{align*}
    \kappa(x)=
    \begin{dcases}
        \kappa_b & x\in D,\\
        \kappa&  x\in\R\setminus D,
    \end{dcases}\quad\text{and}\quad
    \rho(x)=
    \begin{dcases}
        \rho_b & x\in D,\\
        \rho&  x\in\R\setminus D,
    \end{dcases}
\end{align*}
where the constants $\rho_b, \rho, \kappa, \kappa_b \in \R_{>0}$. The wave speeds inside the resonators $D$ and inside the background medium $\R\setminus D$, are denoted respectively by $v_b$ and $v$, the wave numbers respectively by $k_b$ and $k$, and the contrast between the densities of the resonators and the background medium by $\delta$:
\begin{align}
    v_b:=\sqrt{\frac{\kappa_b}{\rho_b}}, \qquad v:=\sqrt{\frac{\kappa}{\rho}},\qquad
    k_b:=\frac{\omega}{v_b},\qquad k:=\frac{\omega}{v},\qquad
    \delta:=\frac{\rho_b}{\rho}.
\end{align}
 
Up to using a Fourier decomposition in time, we can assume that 
the total wave field $u(t,x)$ is time-harmonic:
\begin{align*}
    u(t,x)=\Re ( e^{-\i\omega t}u(x) ),
\end{align*}
for a function $u(x)$ which solves the one-dimensional Helmholtz equations:
\begin{align}
    -\frac{\omega^{2}}{\kappa(x)}u(x)-\frac{\dd}{\dd x}\left( \frac{1}{\rho(x)}\frac{\dd}{\dd
    x}  u(x)\right) =0,\qquad x \in\R.
    \label{eq:time indip Helmholtz}
\end{align}

In this work, we will consider the following variation of \eqref{eq:time indip Helmholtz}:
\begin{align}
    -\frac{\omega^{2}}{\kappa(x)}u(x)-\frac{\dd}{\dd x}\gamma(x)u(x)-\frac{\dd}{\dd x}\left( \frac{1}{\rho(x)}\frac{\dd}{\dd
    x}  u(x)\right) =0,\qquad x \in\R,
    \label{eq: gen Strum-Liouville}
\end{align}
for a piecewise constant coefficient
\begin{align*}
\gamma(x) = \begin{dcases}
    \gamma,\quad x\in D,\\
    0,\quad x \in \R\setminus D.
\end{dcases}
\end{align*}
This new parameter $\gamma$ extends the usual scalar wave equation to a generalised Strum--Liouville equation via the introduction of an imaginary gauge potential \cite{yokomizo.yoda.ea2022NonHermitiana}. Alternatively, one may think about \eqref{eq: gen Strum-Liouville} as a damped wave equation where the damping acts in the space dimension instead of the time dimension.

In these circumstances of piecewise constant material parameters, the wave problem determined by \eqref{eq: gen Strum-Liouville} can be rewritten as the following system of coupled one-dimensional equations:

\begin{align}
    \begin{dcases}
        u\prii(x) + \gamma u\pri(x)+\frac{\omega^2}{v_b^2}u=0, & x\in D,\\
        u\prii(x) + \frac{\omega^2}{v^2}u=0, & x\in\R\setminus D,\\
        u\vert_{\iR}(x^{\iLR}_i) - u\vert_{\iL}(x^{\iLR}_i) = 0, & \text{for all } 1\leq i\leq N,\\
        \left.\frac{\dd u}{\dd x}\right\vert_{\iR}(x^{\iL}_{{i}})=\delta\left.\frac{\dd u}{\dd x}\right\vert_{\iL}(x^{\iL}_{{i}}), & \text{for all } 1\leq i\leq N,\\
        \left.\frac{\dd u}{\dd x}\right\vert_{\iR}(x^{\iR}_{{i}})=\delta\left.\frac{\dd u}{\dd x}\right\vert_{\iR}(x^{\iL}_{{i}}), & \text{for all } 1\leq i\leq N,\\
        \frac{\dd u}{\dd\ \abs{x}} -\i k u = 0, & x\in(-\infty,x_1^{\iL})\cup (x_N^{\iR},\infty).
    \end{dcases}
\label{eq:coupled ods}
\end{align}

We are interested in the resonances $\omega\in\C$ such that \eqref{eq:coupled ods} has a non-trivial solution. We will perform an asymptotic analysis in a high-contrast limit, given by $\delta\to 0$. We will look for the subwavelength modes within this regime, which we characterise by satisfying $\omega \to 0$ as $\delta\to 0$. This limit will recover subwavelength resonances, while keeping the size of the resonators fixed.

In all what follows we will denote by $H^1(D)$ the usual Sobolev spaces of complex-valued functions on $D$.

\begin{definition}[Subwavelength eigenfrequency]
    A (complex) frequency $\omega(\delta) \in \C$, with non-negative real part, satisfying
    \begin{align}
    \omega(\delta) \to 0\quad \text{as}\quad \delta\to 0
    \label{eq: asymptotic behaviour subwavelength resonances}
    \end{align}
    and such that \eqref{eq:coupled ods} admits a non-zero solution $v(\omega,\delta)\in H^1(D)$ for $\omega=\omega(\delta)$ is called a \emph{subwavelength eigenfrequency}. The solution $v(\omega,\delta)$ is called a \emph{subwavelength eigenmode}.
\end{definition}

We can immediately see that $\omega = 0$ is a trivial solution to \eqref{eq:coupled ods} corresponding to an eigenmode which is constant across all the resonators. In what follows, we will restrict attention to other solutions of \eqref{eq:coupled ods}. 

The use of the Dirichlet-to-Neumann map has proven itself to be an effective tool to find solutions to \eqref{eq:coupled ods} that lend themselves to asymptotic analysis in the high-contrast regime. Hereafter, we will adapt the methods of \cite{feppon.cheng.ea2023Subwavelength} to \eqref{eq:coupled ods}.

The first step is to solve the outer problem
\begin{align}
    \begin{dcases}
        w\prii(x) + \frac{\omega^2}{v^2}w=0, & x\in\R\setminus\bigcup_{i=1}^N(x_i^{\iL},x_i^{\iR}),\\
        w(x_i^{\iLR})=f_i^{\iLR}, & \text{for all }1\leq i\leq N,\\
        \frac{\dd u}{\dd\ \abs{x}}u -\i k u = 0, & x\in(-\infty,x_1^{\iL})\cup (x_N^{\iR},\infty),
    \end{dcases}
\label{eq:coupled ods outside}
\end{align}
for some boundary data $f_i^{\iLR}\in \C^{2N}$. Its solution is simply given by 
\begin{align}
w(x)=\begin{dcases}
    f_1^{\iL} e^{-\i k (x-x_1^{\iL})}, & x\leq x_1^{\iL},\\
    a_i e^{\i k x}+b_i e^{-\i k x}, & (x_i^{\iR},x_{i+1}^{\iL}),\\
    f_N^{\iR} e^{-\i k (x-x_N^{\iR})}, & x_1^{\iR}\leq x,\\
\end{dcases}
\label{eq: outer solution}
\end{align}
where $a_i$ and $b_i$ are given by
\begin{align*}
\begin{pmatrix}
    a_i\\b_i
\end{pmatrix} = - \frac{1}{2\i \sin(k s_i)}
\begin{pmatrix}
    e^{-\i k x_{i+1}^{\iL}} & -e^{-\i k x_i^{\iR}} \\
    -e^{\i k x_{i+1}^{\iL}} & e^{\i k x_i^{\iR}} 
\end{pmatrix}
\begin{pmatrix}
    f_i^{\iR}\\ f_{i+1}^{\iL}
\end{pmatrix},
\end{align*}
as shown in \cite[Lemma 2.1]{feppon.cheng.ea2023Subwavelength}. The second and harder step is to combine the expression for the outer solution \eqref{eq: outer solution} with the boundary conditions in order to find the full solution. We will handle this information through the Dirichlet-to-Neumann map.

\begin{definition}[Dirichlet-to-Neumann map]
        For any $k\in\C$ which is not of the form 
          $n\pi/s_{i}$ for some $n\in\Z\backslash\{0\}$ and $1\<i\<N-1$, 
          the \emph{Dirichlet-to-Neumann map} with wave number $k$ is the linear operator 
            $\mathcal{T}^{k}\,:\,\C^{2N}\to \C^{2N}$ defined by 
            \begin{align*}
          \mathcal{T}^{k}[(f_i^{\iLR})_{1 \leq i \leq N}]=
          \left(\pm\frac{\dd w}{\dd x}(x_i^{\iLR})\right)_{1 \leq i \leq N},
            \end{align*}
      where $w$ is  the unique solution to \eqref{eq:coupled ods outside}.
\end{definition}

We refer to \cite[Section 2]{feppon.cheng.ea2023Subwavelength} for a more extensive discussion of this operator, but recall that $\mathcal{T}^{k}$ has a block-diagonal matrix representation
\begin{align}
    \label{eq: Matrix form DTN}
    T^{k}\begin{pmatrix}
        f_1^{\iL}\\
       f_1^{\iR}\\
       \vdots\\
       f_N^{\iL}\\
       f_N^{\iR}
    \end{pmatrix} = \begin{pmatrix}
        \i k &&&&& \\
        & A^{k}(s_{1}) & & & & \\
        & & A^{k}(s_{2}) & & & \\
        &  & &\ddots & & \\
        & & & &  A^{k}(s_{(N-1)}) &\\
        & & & & & \i k \\
    \end{pmatrix}\begin{pmatrix}
        f_1^{\iL}\\
       f_1^{\iR}\\
       \vdots\\
       f_N^{\iL}\\
       f_N^{\iR}
    \end{pmatrix},
\end{align} where, for any real $\ell\in\R$, $A^{k}(\ell)$ denotes the $2\times 2$ symmetric matrix given by
    \begin{equation}
    \label{eqn:1lzi8}
        A^{k}(\ell):=\begin{pmatrix}
            -\dfrac{k \cos(k\ell)}{\sin(k\ell)} & \dfrac{k}{\sin(k\ell)} \\
            \dfrac{k}{\sin(k\ell)} & -\dfrac{k\cos(k\ell)}{\sin(k\ell)}
        \end{pmatrix}.
    \end{equation}
Consequently, $T^{k}$ is holomorphic in $k$ in a neighbourhood of the origin and admits a power series representation $T^{k} = \sum_{n\geq 0}k^n T_n$, where
\begin{align}
T_0 = \begin{pmatrix}
    0 &&&&& \\
    & A^{0}(s_{1}) & & & & \\
    & & A^{0}(s_{2}) & & & \\
    &  & &\ddots & & \\
    & & & &  A^{0}(s_{(N-1)}) &\\
    & & & & & 0 \\
\end{pmatrix}, \label{eq: DTN T0}
\end{align}
and $A^0(s):=\lim_{k\to0} A^k(s)$.

The above properties of the Dirichlet-to-Neumann map will be crucial to find subwavelength eigenfrequencies. We will use $\mathcal{T}^{k}$ and $T^k$ interchangeably.

\subsection{Characterisation of the subwavelength eigenfrequencies}
The Dirichlet-to-Neumann map allows us to reformulate \eqref{eq:coupled ods} as follows:
\begin{align}
    \begin{dcases}
        u\prii(x) + \gamma u\pri(x)+\frac{\omega^2}{v_b^2}u=0, & x\in\bigcup_{i=1}^N(x_i^{\iL},x_i^{\iR}),\\
        \left.\frac{\dd u}{\dd x}\right\vert_{\iR}(x^{\iL}_{{i}})=-\delta\dtn^{\frac{\omega}{v}}[u]^{\iL}_i, & \forall\ 1\leq i\leq N,\\
        \left.\frac{\dd u}{\dd x}\right\vert_{\iL}(x^{\iR}_{{i}})=\delta\dtn^{\frac{\omega}{v}}[u]^{\iR}_i, & \forall\ 1\leq i\leq N.
    \end{dcases}
\label{eq:coupled ods with dtn}
\end{align}
We can then further rewrite \eqref{eq:coupled ods with dtn} in weak form by multiplying it by a test function $w\in H^1(D)$ and integrating on the intervals. Explicitly, we obtain that $u$ is solution to \eqref{eq:coupled ods with dtn} if and only if
\begin{align}
a(u,w)=0 \label{eq:formulation in weak form}
\end{align}
for any $w\in H^1(D)$, where
\begin{align}
    a(u,w) &= \sum_{i=1}^N\int_{x_i^{\iL}}^{x_i^{\iR}} u\pri\bar{w}\pri -\gamma u\pri \bar{w}-\frac{\omega^2}{v_b^2}u\bar{w}\dd x \nonumber\\
    &-\delta \sum_{i=1}^N \bar{w}(x_i^{\iR})\dtn^{\frac{\omega}{v}}[u]_i^{\iR} + \bar{w}(x_i^{\iL})\dtn^{\frac{\omega}{v}}[u]_i^{\iL}.
\end{align}
We also introduce a slightly modified bilinear form
\begin{align}
    a_{\omega,\delta}(u,w) &= \sum_{i=1}^N\left(\int_{x_i^{\iL}}^{x_i^{\iR}} u\pri\bar{w}\pri -\gamma u\pri \bar{w}\dd x + \int_{x_i^{\iL}}^{x_i^{\iR}} u\dd x\int_{x_i^{\iL}}^{x_i^{\iR}} \bar{w}\dd x \right)\nonumber\\
    &-\sum_{i=1}^N\left(\int_{x_i^{\iL}}^{x_i^{\iR}}\frac{\omega^2}{v_b^2}u\bar{w}\dd x + \delta\left(\bar{w}(x_i^{\iR})\dtn^{\frac{\omega}{v}}[u]_i^{\iR} + \bar{w}(x_i^{\iL})\dtn^{\frac{\omega}{v}}[u]_i^{\iL}\right)\right).
    \label{eq: def a weak form}
\end{align}
This new form parametrised by $\omega$ and $\delta$ is an analytic perturbation of the $a_{0,0}$ form, which is continuous coercive on $H^1(D)$ and so $a_{\omega,\delta}$ inherits this property.

We exploit this, by defining the $h_j(\omega,\delta)$ functions as the Lax--Milgram solutions to the variational problem
\begin{align}
    a_{\omega,\delta}(h_j(\omega,\delta),w) = \int_{x_j^{\iL}}^{x_j^{\iR}} \bar{w}\dd x \label{eq: def h_j}
\end{align}
for every $w\in H^1(D)$ and $1\leq j\leq N$. In the following lemma we show that the functions $h_j$ allow us to reduce \eqref{eq:formulation in weak form} to a finite dimensional $N\times N$ linear system by acting as basis functions.

\begin{lemma}\label{lemma: I-capmat=0}
    Let $\omega \in \C$ and $\delta \in \R$ belong to a neighbourhood of zero such that $a_{\omega,\delta}$ is coercive. The variational problem \eqref{eq:formulation in weak form} admits a non-trivial solution $u\equiv u(\omega,\delta)$ if and only if the $N\times N$ non-linear eigenvalue problem
    \begin{align*}
        (I-\exactcapmat(\omega,\delta))\bm x = 0
    \end{align*}
    has a  solution $\omega$ and $\bm x:=(x_i(\omega,\delta))_{1\leq i\leq N}$, where $\exactcapmat(\omega,\delta)$ is the matrix  given by
    \begin{align}
    \label{eq: def exact cap mat}
    \exactcapmat(\omega,\delta)\equiv (\exactcapmat(\omega,\delta)_{ij})_{1\leq i,j\leq N}:= 
        \left( \int_{x_i^{\iL}}^{x_i^{\iR}} h_j(\omega,\delta)\dd x
        \right)_{1\leq i,j\leq N}. 
    \end{align}
\end{lemma}
\begin{proof}
    This can be shown using the arguments presented in \cite[Lemma 3.4]{feppon.cheng.ea2023Subwavelength}.
\end{proof}
Subwavelength eigenfrequencies are thus the values $\omega(\delta)$ satisfying \eqref{eq: asymptotic behaviour subwavelength resonances} for which $I-\exactcapmat(\omega,\delta)$ is not invertible.

\subsection{Discrete approximation of the subwavelength eigenfrequencies and the gauge capacitance matrix}

In the high-contrast, low-frequency scattering problems, recent work has shown that capacitance matrices can be used to provide asymptotic characterisations of the resonant modes \cite{ammari.davies.ea2021Functional}. Capacitance matrices were first introduced to describe many-body electrostatic systems by Maxwell \cite{maxwell1873treatise}, and have recently enjoyed a resurgence in subwavelength physics. In particular, for one-dimensional models like the one considered here, a capacitance matrix formulation has proven to be very effective in both efficiently solving the problem as well as providing valuable insights in the solution \cite{feppon.cheng.ea2023Subwavelength, ammari.barandun.ea2023Edge}. Therefore, we seek a similar formulation for the case of subwavelength resonators with an imaginary gauge potential.

In the case of systems of Hermitian subwavelength resonators (i.e., when $\gamma=0$), the entries of the capacitance matrix are defined by
\begin{align}\label{eq:C}
    \capmat_{i,j} = -\int_{\partial D_i}\frac{\partial V_j}{\partial \nu}\dd \sigma, 
\end{align}
where $\nu$ is the outward-pointing normal, $D_i$ the $i$-th resonator and $V_i : \R\to \R$ the solutions of the problems
\begin{align}
    \begin{dcases}
        -\frac{\dd{^2}}{\dd x^2} V_i =0, & x\in\R\setminus\bigcup_{i=1}^N(x_i^{\iL},x_i^{\iR}), \\
        V_i(x)=\delta_{ij}, & x\in (x_j^{\iL},x_j^{\iR}),\\
        V_i^\alpha(x) = \BO(1) & \mbox{as } \abs{x}\to\infty,
    \end{dcases}
    \label{eq: def V_i}
\end{align}
where $\delta_{ij}$ denotes the Kronecker symbol; see \cite{ammari.davies.ea2021Functional}. 

Here, the formulation for the non-Hermitian system \eqref{eq:coupled ods} is different from the Hermitian case. The following definition is in agreement with the three-dimensional case for the skin effect \cite{3DSkin}. We let $\R^* \coloneqq \R \setminus \{0\}$.

\begin{definition}[Gauge capacitance matrix]
    For $\gamma \in \R^*$,  we define the \emph{gauge capacitance matrix} $\capmatg\in\R^{N\times N}$ by 
    \begin{align}
        \capmat_{i,j}^\gamma \coloneqq -\frac{|D_i|}{\int_{D_i}e^{\gamma x}\dd x}\int_{\partial D_i} e^{\gamma x} \frac{\partial V_j(x)}{\partial \nu}\dd \sigma,
        \label{eq: def cap mat}
    \end{align}
    where $V_j$ is defined by \eqref{eq: def V_i}.
\end{definition}

It is easy to see that the gauge capacitance matrix is tridiagonal, non-symmetric,  and is given by
\begin{align}
    \capmat_{i,j}^\gamma \coloneqq \begin{dcases}
        \frac{\gamma}{s_1} \frac{\ell_1}{1-e^{-\gamma \ell_1}}, & i=j=1,\\
         \frac{\gamma}{s_i} \frac{\ell_i}{1-e^{-\gamma \ell_i}} -\frac{\gamma}{s_{i-1}} \frac{\ell_i}{1-e^{\gamma \ell_i}},  & 1< i=j< N,\\
       - \frac{\gamma}{s_i} \frac{\ell_i}{1-e^{-\gamma \ell_j}},  & 1\leq i=j-1\leq N-1,\\
       \frac{\gamma}{s_j} \frac{\ell_i}{1-e^{\gamma \ell_j}}, & 2\leq i=j+1\leq N,\\
      - \frac{\gamma}{s_{N-1}} \frac{\ell_N}{1-e^{\gamma \ell_N}}, & i=j=N,\\
    \end{dcases}\label{eq: explicit coef cap mat}
\end{align} 
while all the other entries are zero. 

The following proposition is a central result. It shows that to leading order in $\delta$ the gauge capacitance matrix encodes all the information of subwavelength eigenfrequencies.

\begin{proposition}\label{prop: asymptotic of hj}
     Let $\omega \in \C$ and $\delta \in \R$ belong to a small enough neighbourhood of zero. Then, the matrix $\exactcapmat(\omega,\delta)$ from \eqref{eq: def exact cap mat} has the following asymptotic expansion:
    \begin{align}
        \exactcapmat(\omega,\delta) = I+\frac{\omega^2}{v_b^2} V\inv - \delta V\inv\capmatg V\inv + \BO((\omega+\delta)^2).
        \label{eq: integral asymptotic expansion hj}
    \end{align}
    Here, 
    \begin{equation} \label{defV} 
    V = \diag(\ell_1,\dots,\ell_N)
    \end{equation} is the volume matrix.
 \end{proposition}
\begin{proof}
    We need to show that unique solution $h_j(\omega, \delta)$ with $1\leq j\leq N$ to the variational problem \eqref{eq:formulation in weak form} has the given integral asymptotic behaviour as $\omega,\delta \to 0$.
    From the definition of $a_{\omega, \delta}$, the function $h_j\equiv
    h_j^{\alpha}(\omega,\delta)$ satisfies the
    following differential equation written in strong form:
   \begin{equation}
   \label{eq: ODE hj strong}
    \left\{
    \begin{aligned}
     -h_j\prii -\gamma h_j\pri - \frac{\omega^2}{v_b^2} h_j + \sum_{i=1}^{N}
        \left(\int_{x_i^{\iL}}^{x_i^{\iR}} h_j \dd x\right) \mathds{1}_{(x_i^{\iL},x_i^{\iR})}, &= 
        \mathds{1}_{(x_j^{\iL},x_j^{\iR})} &  \text{ in }
        \bigcup_{i=1}^{N}(x_i^{\iL},x_i^{\iR}),\\
        -\frac{\dd h_j}{\dd x}(x_i^{\iL}), &= \delta
        \mathcal{T}^{\frac{\omega}{v}}[h_j]_i^{\iL} & \text{ for all } 1 \leq i \leq
        N,\\
        \frac{\dd h_j}{\dd x}(x_i^{\iR}) &= \delta
        \mathcal{T}^{\frac{\omega}{v}}[h_j]_i^{\iR} & \text{ for all } 1 \leq i \leq
        N.\\
    \end{aligned}
    \right.
    \end{equation}
    Since $\mathcal{T}^{\frac{\omega}{v}}$ is analytic in $\omega^{2}$, it
    follows that $h_j(\omega,\delta)$
    is analytic in $\omega^{2}$ and $\delta$: there exist functions $(h_{j,2p,k})_{p{\geq}0, k{\geq}0}$ such that $h_j(\omega,\delta)$ can be written as the following convergent series in $H^1(D)$:
    \begin{equation}
        h_j(\omega,\delta) = \sum_{p,k=0}^{+\infty} \omega^{2p}\delta^k
        h_{j,2p,k}.
    \end{equation}
    By using the power series expansion of the Dirichlet-to-Neumann map and identifying powers of $\omega$ and
    $\delta$, we obtain the following equations characterising the functions
    $(h_{j,2p,k})_{p{\geq}0, k{\geq}0}$:
   \begin{align}
   \begin{dcases}
       -h_{j,2p,k}\prii -\gamma h_j\pri + \sum_{i=1}^{N}\left(\int_{x_i^{\iL}}^{x_i^{\iR}}
        h_{j,2p,k} \dd x\right)
    \mathds{1}_{(x_i^{\iL},x_i^{\iR})} =  \frac{1}{v_j^2} h_{j,2p-2,k} +
        \mathds{1}_{(x_j^{\iL},x_j^{\iR})}  \delta_{0p} \delta_{0k}  & \text{ in }
      D,
        \\ -\frac{\dd
        h_{j,2p,k}}{\dd x}(x_i^{\iL}) = \sum_{n=0}^p \frac{1}{v^{2n}}
        \mathcal{T}_{2n}[h_{j,2p-2n,k-1}]_i^{\iL}, &1 \leq i \leq N,\\
        \frac{\dd
        h_{j,2p,k}}{\dd x}(x_i^{\iR}) = \sum_{n=0}^p \frac{1}{v^{2n}}
        \mathcal{T}_{2n}[h_{j,2p-2n,k-1}]_i^{\iR}, &1 \leq i \leq N,\\
        \end{dcases}
    \end{align}
   with the convention  that $h_{j,2p,k}=0$ for negative indices $p$ and $k$. 
    It can then be easily obtained by induction that
    \begin{align*}
        h_{j,2p,0} = \frac{\mathds{1}_{(x_j^{\iL},x_j^{\iR})}}{v_j^{2p} \ell_j^{p+1}}\quad \text{ for
        any } p \geq 0,\quad 1\leq j\leq N.
    \end{align*}
   Then, for $p=0$ and $k=1$, we find that $h_{j,0,1}$ satisfies
   \begin{equation}
   \label{eq:ODEhp0k1}\left\{
   \begin{aligned}
     -h_{j,0,1}\prii  -\gamma h_{j,0,1}\pri +\sum_{i=1}^{N} \left(\int_{x_i^{\iL}}^{x_i^{\iR}}
       h_{j,0,1} \dd x\right) \mathds{1}_{(x_i^{\iL},x_i^{\iR})} &= 0    & \text{ in }
       D,\\
    -\frac{\dd h_{j,0,1}}{\dd x}(x_i^{\iL}) &= \mathcal{T}_0[h_{j,0,0}]_i^{\iL} & \text{ for all } 1 \leq i \leq N, \\
    \frac{\dd h_{j,0,1}}{\dd x}(x_i^{\iR}) &= \mathcal{T}_0[h_{j,0,0}]_i^{\iR} & \text{ for all } 1 \leq i \leq N. \\
    \end{aligned}
    \right.
   \end{equation}
    From \eqref{eq: DTN T0} with
    $f_{i}^{\iLR}:=h_{j,0,0}(x_i^{\iLR})=\delta_{ij}/\ell_j$ for $1\leq i\leq N$, we obtain
    \begin{align}
        \label{eq:bdy condition hj01}
        \left\{\begin{aligned}
           \mathcal{T}_0[h_{j,0,0}]_1^{\iL}
           &=0,&&\\
           \mathcal{T}_0[h_{j,0,0}]_i^{\iL} &
           =-\frac{1}{\ell_j}\frac{1}{s_{ i-1 }}\left(\delta_{ij}-\delta_{(i-1)j}\right)
           &&\text{
       for }2\leq i\leq N,\\
   \mathcal{T}_0[h_{j,0,0}]_i^{\iR}& =\frac{1}{\ell_j}\frac{1}{s_{i}}\left(\delta_{(i+1)j}-\delta_{ij}\right)
           &&            \text{ for }1\leq i\leq N-1,\\
           \mathcal{T}_0[h_{j,0,0}]_N^{\iR}
           &=0 .&& 
       \end{aligned}\right.
       \end{align}
    The general solution for the ODE
    $$
    -y\prii-ay\pri + b =0
    $$
    is given by
    $$
    y(x) = \frac{\kappa_1 e^{-ax}}{a} + \frac{b x}{a} + \kappa_2
    $$
    for two constants $\kappa_1$ and $\kappa_2$. So the solution to \eqref{eq:ODEhp0k1} on $\mathds{1}_{(x_i^{\iL},x_i^{\iR})}$ is given by
    \begin{align*}
        h_{j,0,1} = \frac{\kappa_1 e^{-\gamma x}}{\gamma} + \frac{x}{\gamma}\int_{x_i^{\iL}}^{x_i^{\iR}}
        h_{j,0,1} \dd x + \kappa_2,
    \end{align*}
    having derivative
    \begin{align*}
        \frac{\dd}{\dd x}h_{j,0,1} = -\kappa_1 e^{-\gamma x} + \frac{1}{\gamma}\int_{x_i^{\iL}}^{x_i^{\iR}}
        h_{j,0,1} \dd x.
    \end{align*}
    Imposing the transmission conditions on the derivatives, 
    \begin{align*}
        -\kappa_1 e^{-\gamma x_i^{\iL}} + \frac{1}{\gamma}\int_{x_i^{\iL}}^{x_i^{\iR}}
        h_{j,0,1} \dd x = -\mathcal{T}_0[h_{j,0,0}]_i^{\iL},\\
        -\kappa_1 e^{-\gamma x_i^{\iR}} + \frac{1}{\gamma}\int_{x_i^{\iL}}^{x_i^{\iR}}
        h_{j,0,1} \dd x = \mathcal{T}_0[h_{j,0,0}]_i^{\iR},\\
    \end{align*}
    we obtain a value for the integral 
    \begin{align*}
        \int_{x_i^{\iL}}^{x_i^{\iR}}
        h_{j,0,1} \dd x = \gamma\frac{\mathcal{T}_0[h_{j,0,0}]_i^{\iR}e^{\gamma\ell_i} - \mathcal{T}_0[h_{j,0,0}]_i^{\iL}}{(1-e^{\gamma\ell_i})},
    \end{align*}
    which combined with \eqref{eq:bdy condition hj01} yields the result after some careful algebraic manipulation.
\end{proof}

The following corollary, whose proof is similar to the analogous results in \cite{ammari.barandun.ea2023Edge, feppon.ammari2022Modal}, describes how the gauge capacitance matrix characterises the non-trivial solutions to \eqref{eq:coupled ods}.
\begin{corollary}[Discrete approximations of the eigenfrequencies and eigenmodes]\label{cor: approx via eva eve}
    The $N$ subwavelength eigenfrequencies $\omega_i$ satisfy, as $\delta\to0$,
    \begin{align*}
        \omega_i =  v_b \sqrt{\delta\lambda_i} + \BO(\delta),
    \end{align*}
    where $(\lambda_i)_{1\leq i\leq N}$ are the eigenvalues of the eigenvalue problem
\begin{equation}
\label{eq:eigevalue problem capacitance matrix}
\capmatg\bm a_i = \lambda_i V \bm a_i,\qquad 1\leq i\leq N.
\end{equation}
Furthermore, let $u_i$ be a subwavelength eigenmode corresponding to $\omega_i$ and let $\bm a_i$ be the corresponding eigenvector of $\capmatg$. Then
        \begin{align*}
            u_i(x) = \sum_j \bm a_i^{(j)}V_j(x) + \BO(\delta),
        \end{align*}
        where $V_j$ are defined by \eqref{eq: def V_i} and $\bm a^{(j)}$ denotes the $j$-th entry of the eigenvector.
\end{corollary}

\begin{lemma}[Properties of the gauge capacitance matrix]\hfill\\
    \begin{enumerate}[(i)]
        \item Recall the Hermitian capacitance matrix $\capmat$, defined in \eqref{eq:C}. Then, for any $1\leq i,j\leq N$ it holds that
        \begin{align*}
        \lim_{\gamma\to 0} \capmat_{i,j}^\gamma = \capmat_{i,j}; 
        \end{align*}
        \item For equally spaced identical resonators: $s_i=s,\ \ell_i=\ell$ for all $1\leq i\leq N$, the gauge capacitance matrix is \emph{almost} Toeplitz taking the form\begin{align}
        \label{eq: cap mat ESI}
            \capmat_{i,j}^\gamma \coloneqq \begin{dcases}
                \frac{\gamma \ell}{s} \frac{1}{1-e^{-\gamma\ell}}, & i=j=1,\\
                \frac{\gamma \ell}{s} \coth(\gamma\ell/2), & 1< i=j< N,\\
                \pm\frac{\gamma \ell}{ s} \frac{1}{1-e^{\pm\gamma\ell}}, & 1\leq i=j \pm 1\leq N,\\
                -\frac{\gamma \ell}{ s} \frac{1}{1-e^{\gamma\ell}}, & i=j=N.\\
            \end{dcases}
            \end{align} \label{cor:item: cap mat for esi}
    \end{enumerate}
    \label{cor: properties of cap mat}
\end{lemma}

\section{Skin effect and localised interface modes}\label{sec: skin effect}
We now turn to the question of eigenmode condensation. We will begin by proving that the skin effect occurs in a finite system of subwavelength resonators. Then, once this fundamental result has been established, we will consider more exotic structures. For example, we will consider a structure with an interface formed by adjoining two half-structures with opposite signs of $\gamma$, leading to wave localisation along the interface, as shown \cref{sec:interface}.

\subsection{Skin effect}
Using the extensive theory of Toeplitz matrices and operators, we will now show that the system \cref{eq:coupled ods} is a prototypical example of the non-Hermitian skin effect. We will consider a system of equally spaced identical resonators as described in \cref{cor: properties of cap mat}~(iii). That is, a chain of $N$ resonators with $s_i=s$ and $\ell_i=\ell$ for all $1\leq i\leq N$. For this case, the particular structure of the capacitance matrix allows us to apply existing results concerning the spectra of tridiagonal Toeplitz matrices, which we briefly recall in \cref{sec: teoplitz theory}. We will be able to derive a variety of explicit results that would not be possible for more complex situations. Nevertheless, \cref{cor: approx via eva eve} applies for more general cases and the numerical results (which we present for dimers in \cref{sec: dimers}) show that similar results hold for systems with multiple resonators in the unit cell.

Applying the explicit formula for the eigenpairs of perturbed tridiagonal Toeplitz matrices from \cref{lemma: eigenpairs tilde T n} to $\capmatg$, we obtain the 
following theorem.
\begin{theorem} \label{thm:decay}
The eigenvalues of $\capmatg$  are given by
\begin{align}
    \lambda_1 &= 0,\nonumber\\
    \lambda_k &= \frac{\gamma}{s} \coth(\gamma\ell/2)+\frac{2\abs{\gamma}}{s}\frac{e^{\frac{\gamma\ell}{2}}}{\vert e^{\gamma\ell}-1\vert}\cos\left(\frac{\pi}{N}k\right), \quad 2\leq k\leq N . \label{eq: eigenvalues capmat}
\end{align}
Furthermore, the associated eigenvectors $\bm a_k$ satisfy the following inequality, for $2\leq k\leq N$
\begin{align}
\vert \bm a_k^{(i)}\vert \leq \kappa_k e^{-\gamma\ell\frac{i-1}{2}}\quad \text{for all } 1\leq i\leq N \label{eq: decay for eigemodes},
\end{align}
for some $\kappa_k\leq (1+e^{\frac{\gamma\ell}{2}})^2$. 
\end{theorem}

Here, $\bm a_k^{(i)}$ denotes the $i$-th component of the $k$-th eigenvector. It is easy to show that the first eigenvector $\bm a_1$ is a constant vector (i.e. $\bm a_1^{(i)}$ is the same for all $i$). From \cref{thm:decay}, we can see that the eigenmodes display exponential decay both with respect to the site index $i$ and the factor $\gamma$. We show this decay in a simulation of a large system of resonators in \cref{fig:skin modes}. Furthermore, \eqref{eq: decay for eigemodes} shows that changing the sign of $\gamma$ swaps the edge at which the localisation occurs.

\begin{figure}[htb]
    \includegraphics[width=\textwidth]{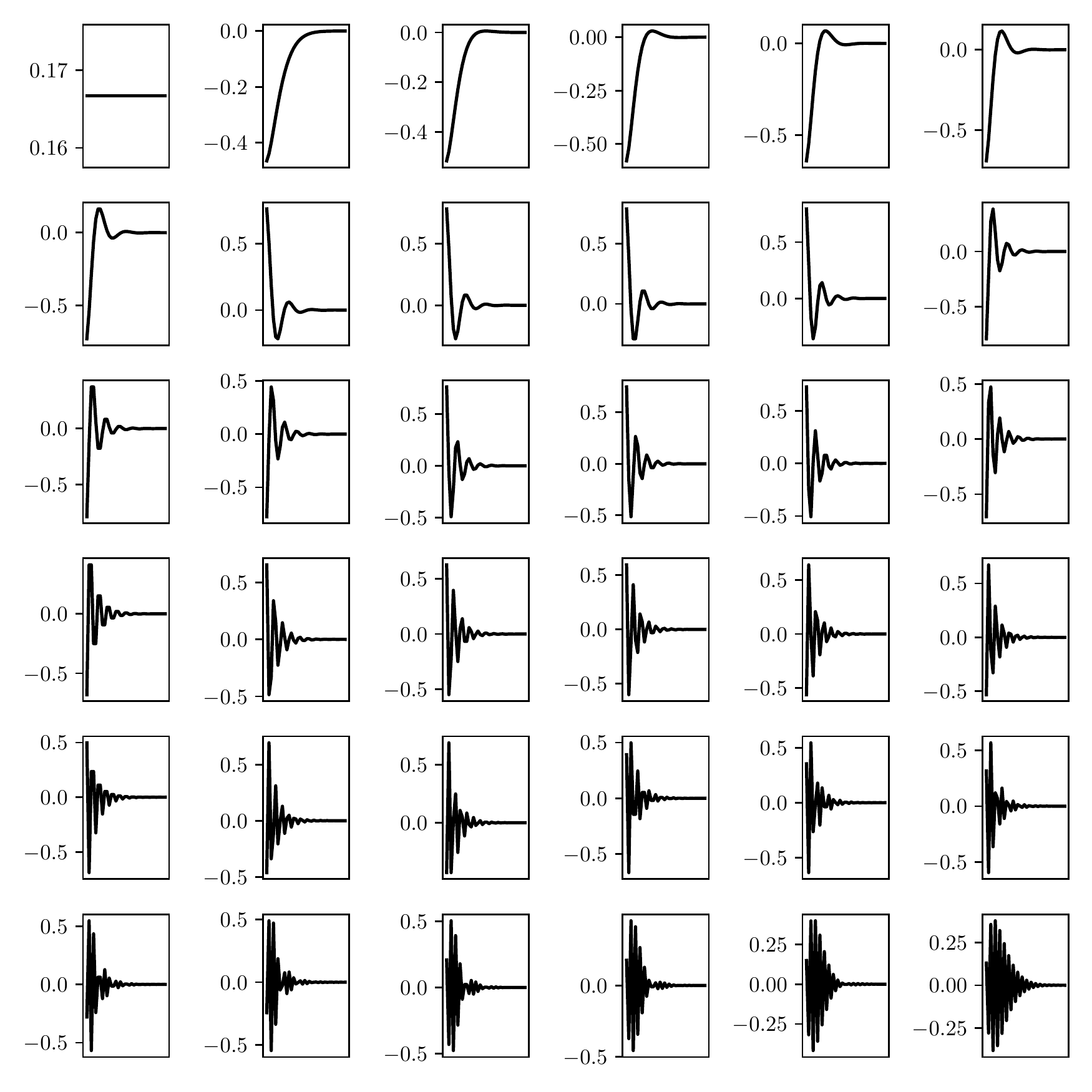}
    \caption{Eigenvector localisation for a system of $N=36$ resonators with $s=\ell=1$ and $\gamma=0.5$. Following \eqref{eq: decay for eigemodes} the eigenmodes present an exponential decay and are thus localised on the left. The top left eigenmode is associated to the trivial eigenvalue 0.} 
    \label{fig:skin modes}
\end{figure}

Of particular interest is the application of Lemmas~\ref{lemma: specturm of toeplitz operator}, \ref{lemma: convergence of speudospectrum topelitz op} and \ref{lemma: facts on eigenvectors teopliz} to the gauge capacitance matrix. The combination of these lemmas shows that the localisation of the eigenvectors of both the finite and semi-infinite capacitance matrix depends on the Fredholm index of the symbol of the associated Toeplitz operator at the corresponding eigenvalue or, equivalently, of its winding number. Specifically, every eigenvector having the corresponding eigenvalue laying the region of the complex plane where the Fredholm index (or equivalently, the winding number of the symbol) is negative, has to be localised.
 This shows the intrinsic topological nature of the skin effect. 

In \cref{fig: winding and pseudospecturm} we display the convergence of the pseudospectrum and the topologically protected region of negative winding number. We remark that the trivial eigenvalue $0$ is outside of the region as the winding number is not defined there. As a consequence, the corresponding eigenvector is not localised.
\begin{figure}[htb]
    \centering
    \begin{subfigure}[t]{0.48\textwidth}
    \centering
    \includegraphics[width=\textwidth]{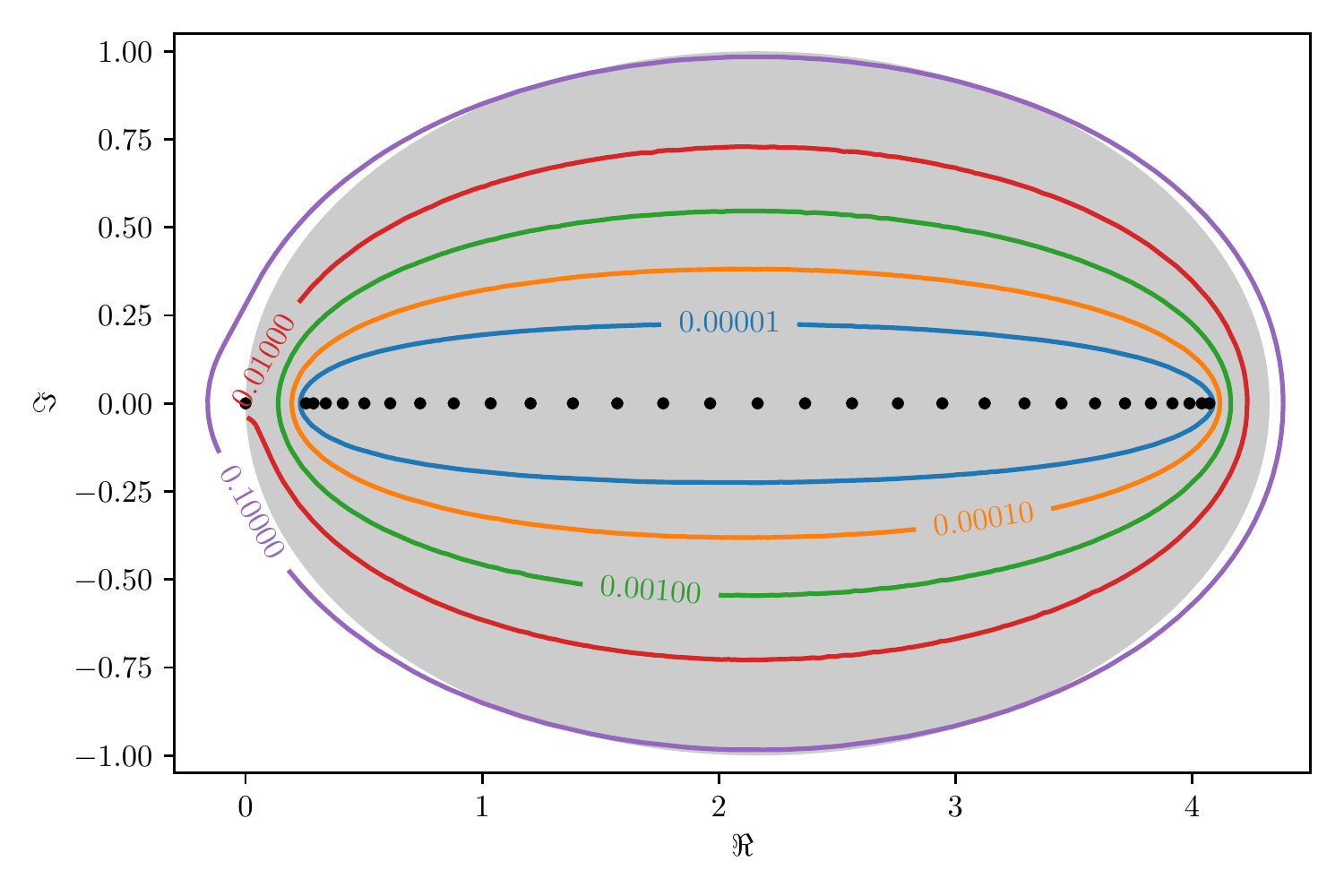}
    \caption{$N=30$.}
    \label{fig:pseudo30}
    \end{subfigure}
    \hfill
    \begin{subfigure}[t]{0.48\textwidth}
    \centering
    \includegraphics[width=\textwidth]{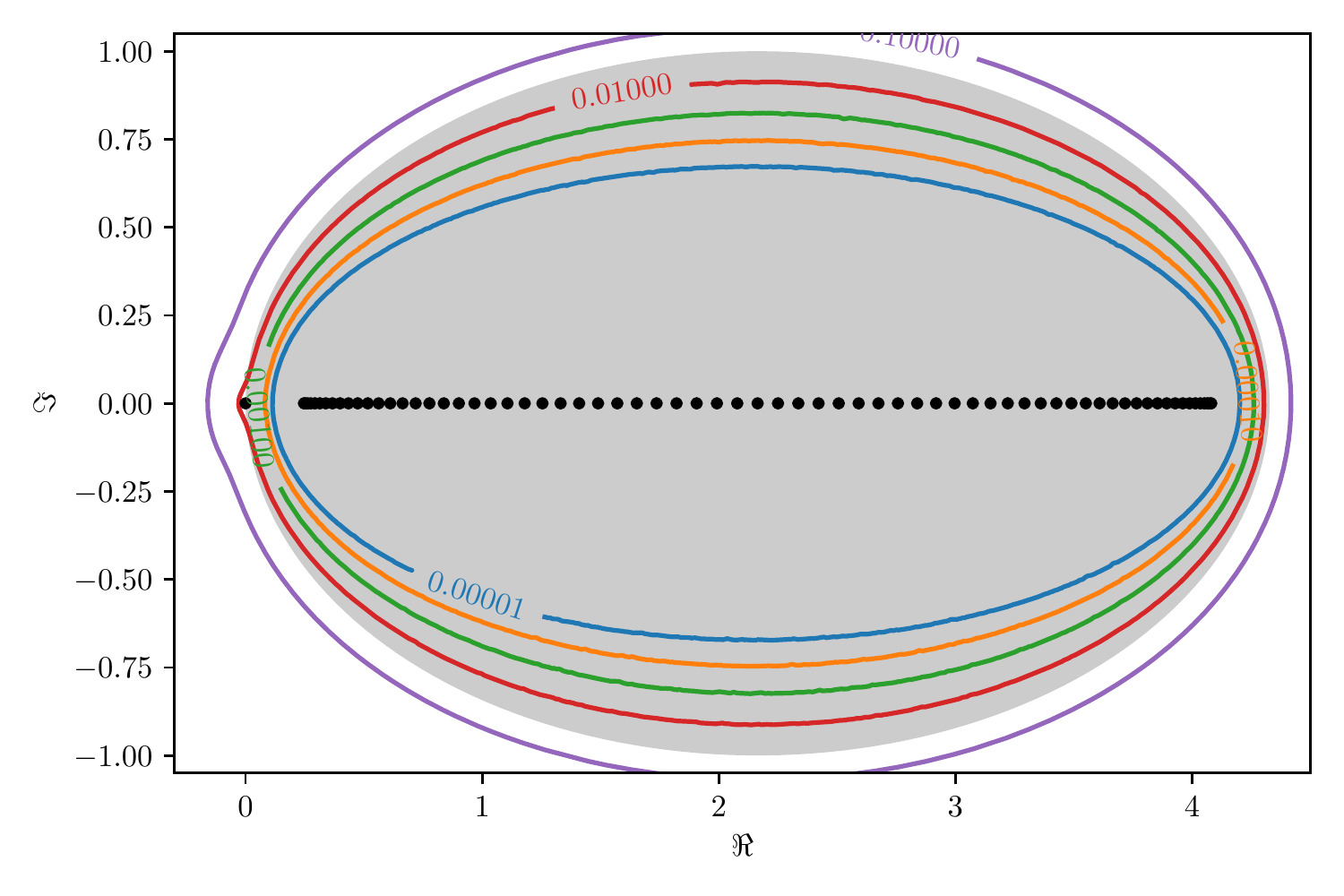}
    \caption{$N=70$.}
    \label{fig:pseudo70}
    \end{subfigure}
\caption{$\epsilon$-pseudospectra of $\capmatg$ for $\epsilon = 10^{-k}$ for $k=1,\dots,5$. Shaded in grey is 
the region where the symbol $f_T$ has negative winding, for $T$ the Toeplitz operator corresponding to the semi-infinite structure. The pseudospectra are computed using \cite{gaulandrePseudoPy}.}
\label{fig: winding and pseudospecturm}
\end{figure}

The explicit formula for eigenvalues from \cref{lemma: eigenpairs tilde T n} also gives some insight on the distribution of the \emph{density of states}. The property of the cosine to have minimal slope when it is close to its maxima and minima, causes eigenvalues to cluster at the edges of the range of the spectrum. This is demonstrated by the non-uniform distribution of the black dots in \cref{fig: winding and pseudospecturm}.

We conclude this section with a qualitative analysis of the spectral decomposition of $\capmatg$. In \cref{fig: eignemode localisation} we show the eigenvector localisation as a function of the site index for finite arrays of various sizes. The localisation of a vector $v$ is measured using the quantity $\|v\|_\infty/\|v\|_2$. After rescaling the site index, we expect there to be some invariance to the array size $N$, based on the formulas of \cref{lemma: eigenpairs tilde T n}. \cref{fig: singular values decay} shows, on the other hand, the singular values of the eigenvector matrix, again with the indices normalised. The number of non-zero singular values is a proxy for the dimension of the range of the matrix. This has an exponential dependence on $N$. As a result, \cref{fig: singular val and localisation decay} shows that as $N$ increases, despite the number of eigenvalues growing like $N$, the rank of the matrix of eigenvectors
is fixed. As a result, we have an increasing number of collinear eigenvectors. This can be interpreted as there existing an exceptional point of ``infinite'' order in the arbitrarily large system.

\begin{figure}[htb]
    \centering
    \begin{subfigure}[t]{0.48\textwidth}
    \centering
    \includegraphics[width=\textwidth]{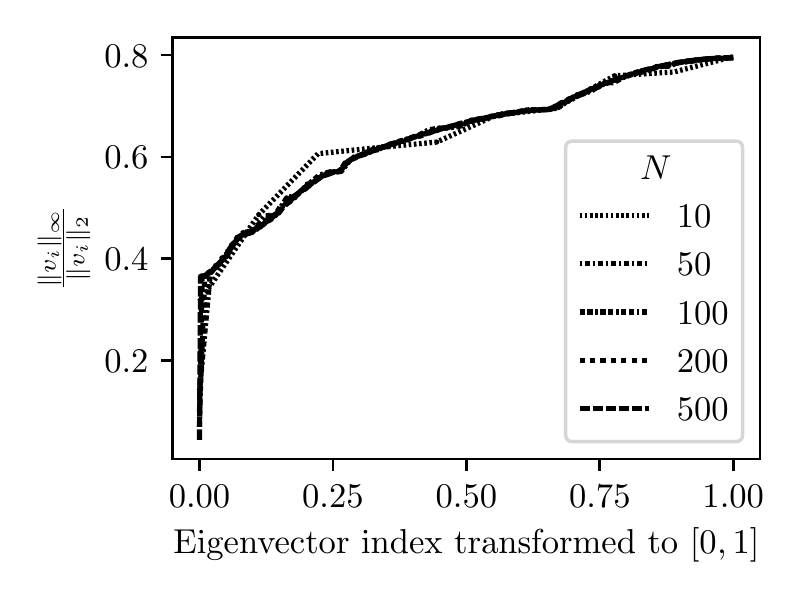}
    \caption{The eigenvector localisation $\frac{\vert v_i\vert_\infty}{\vert v_i\vert_2}$ does not depend on $N$ after rescaling.}
    \label{fig: eignemode localisation}
    \end{subfigure}
    \hfill
    \begin{subfigure}[t]{0.48\textwidth}
    \centering
    \includegraphics[width=\textwidth]{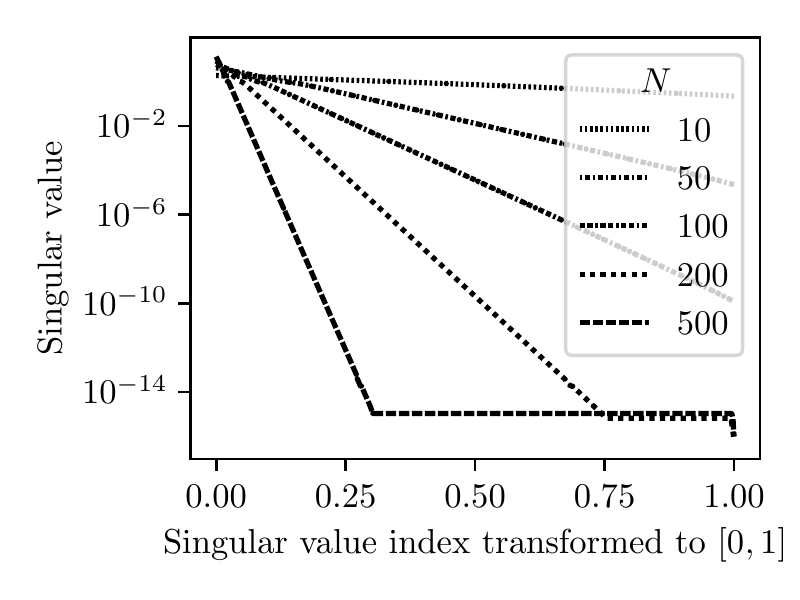}
    \caption{The singular values of the eigenvector matrix decay exponentially in $N$. This shows that the rank of the eigenvector matrix is independent of the size.}
    \label{fig: singular values decay}
    \end{subfigure}
\caption{As $N\to \infty$, the finite structure exhibits an exceptional point of ``infinite'' order: arbitrarily many eigenmodes become close to parallel. Here shown in the case $s=\ell=1$ and $\gamma=0.5$.}
\label{fig: singular val and localisation decay}
\end{figure}

\subsection{Non-Hermitian interface modes between opposing signs of $\gamma$}\label{sec:interface}
We now consider interface modes between two structures where the sign of $\gamma$ is switched from negative (on the left part) to positive (on the right part). Most commonly, localised interface modes are formed by creating a defect in the system's geometric periodicity (see, for example, \cite{ammari.davies.ea2020Topologically}). In non-Hermitian systems based on complex material parameters, similar localised interface modes have been shown to exist in the presence of a defect in the periodicity of the material parameters \cite{ammari.barandun.ea2023Edge,ammari.hiltunen2020Edge}. Given the existence of the skin effect, demonstrated in the previous section, it is reasonable to expect that we might be able to produce a similar localisation effect using systems of resonators with imaginary gauge potential. With this in mind, we consider the following system of $N=2n+1$ resonators:

\begin{align}
    \begin{dcases}
        u\prii(x) \bm+ \gamma u\pri(x)+\frac{\omega^2}{v_b^2}u=0, & x\in\bigcup_{i=1}^{n}(x_i^{\iL},x_i^{\iR}),\\
        u\prii(x) \bm- \gamma u\pri(x)+\frac{\omega^2}{v_b^2}u=0, & x\in\bigcup_{n+1}^N(x_i^{\iL},x_i^{\iR}),\\
        u\prii(x) + \frac{\omega^2}{v^2}u=0, & x\in\R\setminus\bigcup_{i=1}^N(x_i^{\iL},x_i^{\iR}),\\
        u\vert_{\iR}(x^{\iLR}_i) - u\vert_{\iL}(x^{\iLR}_i) = 0, & \text{for all } 1\leq i\leq N,\\
        \left.\frac{\dd u}{\dd x}\right\vert_{\iR}(x^{\iL}_{{i}})=\delta\left.\frac{\dd u}{\dd x}\right\vert_{\iL}(x^{\iL}_{{i}}), & \text{for all } 1\leq i\leq N,\\
        \left.\frac{\dd u}{\dd x}\right\vert_{\iR}(x^{\iR}_{{i}})=\delta\left.\frac{\dd u}{\dd x}\right\vert_{\iR}(x^{\iL}_{{i}}), & \text{for all } 1\leq i\leq N,\\
        \frac{\dd u}{\dd\ \abs{x}}u -\i k u = 0, & x\in(-\infty,x_1^{\iL})\cup (x_N^{\iR},\infty).
    \end{dcases}
\label{eq:coupled ods different gammas}
\end{align}

It is not difficult to see that also with this system we can recover a capacitance matrix for which a similar result as \cref{cor: approx via eva eve} holds. In particular, generalising \eqref{eq: def cap mat}, we get
\begin{align*}
\capmat_{i,j}^{\gamma,-\gamma} =
\begin{dcases}
    \capmat_{i,j}^{\gamma},\quad i \leq n,\\
    \capmat_{i,j}^{-\gamma},\quad i \geq n + 1.
\end{dcases}
\end{align*}

The decay properties of the eigenvectors \eqref{eq: decay for eigemodes} and the symmetry property with respect to $\gamma$ show that this symmetric system \eqref{eq:coupled ods different gammas} has all but two modes localised at the interface. These interface modes are shown in \cref{fig: double_gamma}, superimposed on one another to portray the general trend. 

\begin{figure}[htb]
    
    \centering
    \includegraphics[width=0.5\textwidth]{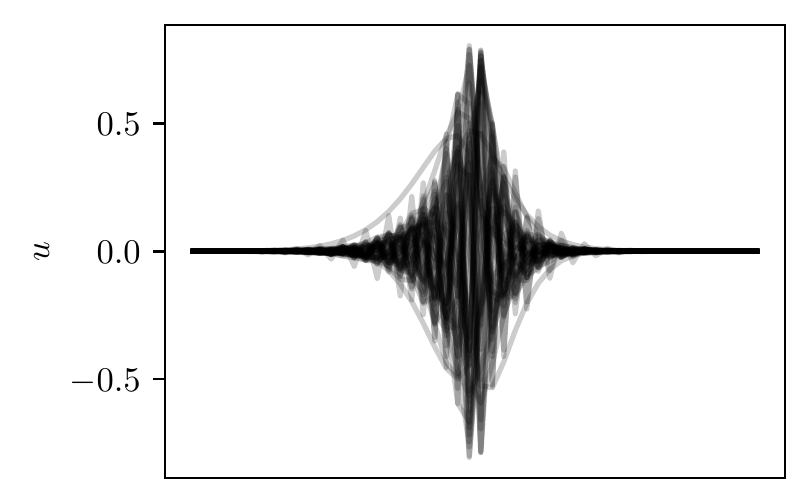}
    
\caption{Plot of all the eigenmodes localised at the interface associated to a system described by \eqref{eq:coupled ods different gammas}. The $x$-axis encodes the site index of the resonators. Simulation performed with a structure of $N=50$ resonators, $\ell=s=1$ and $\gamma=1$. Two trivial eigenmodes are not shown.}
\label{fig: double_gamma}
\end{figure}

\section{Infinite periodic case}\label{sec: periodic case}  %

For the infinite periodic case we consider a one-dimensional system constituted of $N$ periodically repeated disjoint subwavelength resonators. We assume that for some $\ell>0$, $\ell_i = \ell$ for all $i\in\Z$. We use the same notation as in \cite{ammari.barandun.ea2023Edge}. Specifically, $L$ is the size of the unit cell and $\alpha$ denotes the quasiperiodicity parameter of the Floquet-Bloch transform.

\subsection{Standard Brillouin zone} \label{sec: periodic case1}
Similar to the finite case, we are first interested in resonant frequencies $\omega^\alpha(\delta)$ for any $\alpha$ in the first Brillouin zone $Y^*\coloneqq (-\frac{\pi}{L},\frac{\pi}{L}]$. In the periodic case, we call $\alpha\mapsto \omega^\alpha(\delta)$ band functions.

In order to get the quasiperiodic capacitance matrix one needs to apply similar modification as in the finite case to the procedure of \cite{ammari.barandun.ea2023Edge}, where the reader can find the details of the derivation.

\begin{definition}[Quasiperiodic gauge capacitance matrix]
    For $\gamma \in \R^*$, we define the \emph{quasiperiodic gauge capacitance matrix} $\qpcapmatg\in\R^{N\times N}$ by
    \begin{align}
        \capmat_{i,j}^{\gamma,\alpha} &\coloneqq \left(\frac{\gamma}{s_i} \frac{\ell_i}{1-e^{-\gamma \ell_i}} - \frac{\gamma}{s_{i-1}} \frac{\ell_i}{1-e^{\gamma \ell_i}}\right)\delta_{ij} 
        + \left(-\frac{\gamma}{s_i} \frac{\ell_j}{1-e^{-\gamma \ell_j}}\right) \delta_{i(j-1)}
        + \left(\frac{\gamma}{s_j} \frac{\ell_j}{1-e^{\gamma \ell_j}}\right)\delta_{i(j+1)}\nonumber\\
        &+ \left(-e^{-\i \alpha L}\frac{\gamma}{s_N} \frac{\ell_1}{(1-e^{-\gamma \ell_1})}\right) \delta_{iN}\delta_{j1}
        + \left(e^{\i \alpha L}\frac{\gamma}{s_N} \frac{\ell_N}{(1-e^{\gamma \ell_N})}\right) \delta_{i1}\delta_{jN}
        .\label{eq: def qp cap mat}
    \end{align}
\end{definition}
The following results hold. 
\begin{proposition}
    \label{prop: asymptotic expansion of freqs via eigenvals}
     Assume that the eigenvalues of $\qpcapmatg$ are simple. Then the $N$ subwavelength band functions
     $(\alpha \mapsto \omega^{\alpha}_i)_{1\leq i\leq N}$ satisfy, as $\delta\to0$,
     \begin{align*}
         \omega_i^{\alpha} =  v_b\sqrt{\delta\lambda_i^{\alpha}} + \BO(\delta),
     \end{align*}
     where $(\lambda_i^{\alpha})_{1\leq i\leq N}$ are the eigenvalues of the eigenvalue problem
 \begin{equation}
 \qpcapmatg\bm a_i = \lambda_i^{\alpha} V \bm a_i,\qquad 1\leq i\leq N,
 \end{equation}
 where $V$ is defined by (\ref{defV}). We select the $N$ values of $\pm\sqrt{\delta\lambda_i}$ having positive real parts.
 \end{proposition}

\begin{lemma} \label{lemmaN=2}
    For $N=2$ and $\ell_1=\ell_2= \ell$, the eigenvalues of $\qpcapmatg$ are given by 
    \begin{align*}
        \lambda_1^\alpha = \frac{\gamma \left(\left(e^{\gamma }+1\right)
         (s_1+s_2)-\sqrt{2} e^{\gamma/2}
        \sqrt{\left(4s_1s_2
        \cosh (\gamma - \i \alpha  L)+d\right)}\right)
        }{2 \left(e^{\gamma }-1\right) s_1s_2} \ell\\
        \lambda_2^\alpha = \frac{\gamma  
        \left(\left(e^{\gamma }+1\right) 
        (s_1+s_2)+\sqrt{2} e^{\gamma/2} \sqrt{
        \left(4 s_1 s_2 \cosh (\gamma -\i \alpha  L)+d\right)}\right)
        }{2 \left(e^{\gamma }-1\right) s_1 s_2} \ell,
    \end{align*}
    where $d = \cosh (\gamma)(s_1-s_2)^2+(s_1+s_2)^2$. On the other hand, for $N=1$ (i.e., when $s_1=s_2=s$ and $\ell_1=\ell_2=\ell; L=\ell+s$), we have
    \begin{align} \label{4.2b}
        \lambda_1^\alpha = \lambda_2^\alpha = \frac{\gamma \left(e^{\gamma} - e^{\i L \alpha} + 1\right) - e^{\gamma}e^{- \i L \alpha}}{s \left(e^{\gamma} - 1\right)} \ell,
    \end{align}
    and $ \qpcapmatg$ can be considered as a scalar. 
\end{lemma}

The band theory of $\qpcapmatg$ is very rich due to the non-Hermitian structure of the matrix. Indeed one might rapidly see that for some general $1\leq i,j\leq N$
\begin{align*}
    \capmat_{i,j}^{\gamma,\alpha} \neq \overline{\capmat_{j,i}^{\gamma,\alpha}}.
\end{align*}

From now we will consider systems of periodically repeated dimers, that is $N=2$, and $\ell_i=\ell=1$. The band functions $\omega^\alpha$ present very interesting symmetries as the following lemma shows. This behaviour is observed in \cref{fig: bands and eigenvalue winding}, where we see the symmetry of the real parts and antisymmetry of the imaginary parts.

\begin{lemma}\label{lemma: capmat -alpha is conjugation}
    Let $\gamma\in\R^*$. Then, 
    \begin{align*}
        \capmat^{\gamma,-\alpha} = \overline{\capmat^{\gamma,\alpha}}.
    \end{align*}
    Therefore, the real part of the eigenvalues is symmetric and the imaginary part is antisymmetric with respect to $\alpha\mapsto -\alpha$. 
\end{lemma}

\begin{figure}[htb]
    \centering
    \begin{subfigure}[t]{0.48\textwidth}
        \centering
        \includegraphics[width=\textwidth]{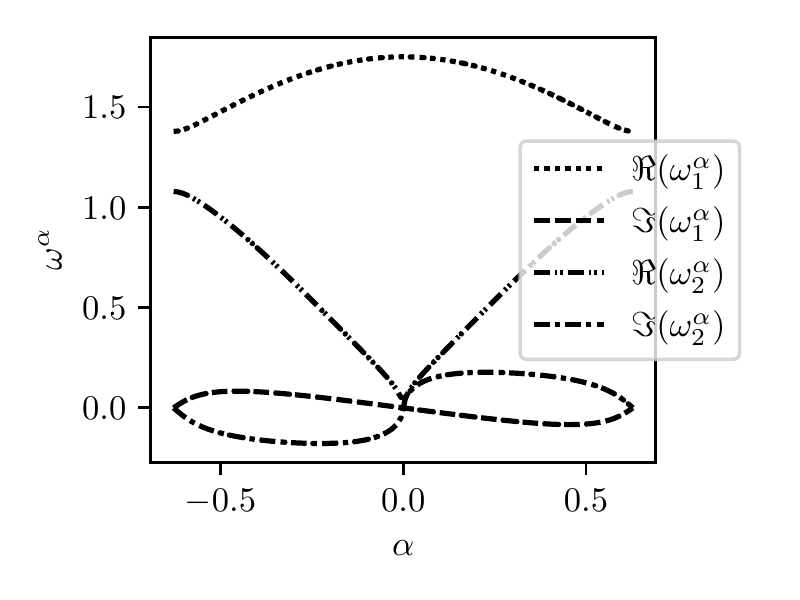}
        \caption{Real and imaginary part of $\omega_i^\alpha$ for $i=1,2$. We observe the symmetries of \cref{lemma: capmat -alpha is conjugation}: the real parts are symmetric and the imaginary parts antisymmetric with respect to $\alpha\mapsto -\alpha$.}
        \label{fig: re and im of band}
    \end{subfigure}
    \hfill
    \begin{subfigure}[t]{0.48\textwidth}
        \centering
        \includegraphics[width=\textwidth]{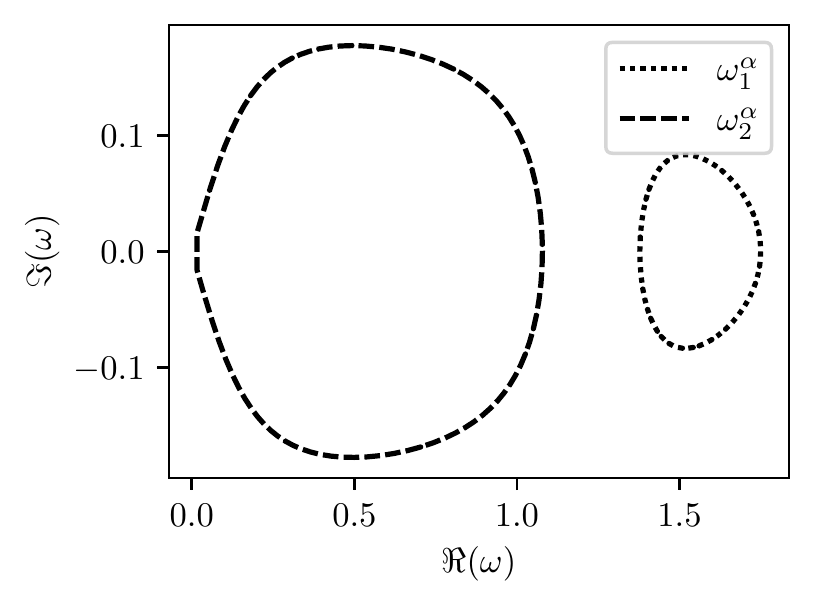}
        \caption{Bands $\alpha\mapsto\omega_i^\alpha$ plotted in complex plane.}
        \label{fig: eigenvalue winding in C}
    \end{subfigure}
    \caption{Band functions of a periodically repeated dimer with $s_1=1$,  $s_2=2$ and $\gamma=0.5$. For both $\omega_1^\alpha$ and $\omega_2^\alpha$ we observe points $E_{1,2}\in\C$ with non-zero winding (see \cref{prop: winding of eigenvalues}).}
    \label{fig: bands and eigenvalue winding}
\end{figure}

The following lemma establishes the exceptional points of the system, where 
the eigenvalues and the eigenvectors of the quasiperiodic gauge capacitance matrix simultaneously coalesce.

\begin{lemma}\label{lemma: critical gamma}
    Let $N=2$ and $0<\gamma$. Then the two eigenvalues of $\qpcapmatg$ coalesce if and only if $\alpha=\pm\frac{\pi}{L}$ and
    \begin{align}
        \cosh (\gamma) \left(s_1^2-6 s_1 s_2+s_2^2\right)+(s_1+s_2)^2 = 0.
        \label{eq: critical gamma}
    \end{align}
    The geometrical multiplicity of the eigenvalue is then $1$.
\end{lemma}
\begin{proof}
    The eigenvalues coalesce if and only if $\tr(\qpcapmatg)^2-4\det(\qpcapmatg)=0$, which reads
    \begin{align*}
        \frac{\gamma ^2 \csch^2\left(\frac{\gamma }{2}\right) \left(4 s_1 s_2 \cosh (\gamma -\i \alpha  L)+\cosh (\gamma )
        (s_1-s_2)^2+(s_1+s_2)^2\right)}{2 s_1^2 s_2^2} = 0.
    \end{align*}
    Comparing real and imaginary parts, we recover the sought formula.

    In order to establish the geometrical multiplicity of the eigenvalue, which reads $\tr(\capmat^{\gamma,\frac{\pi}{L}})/2$, we need to find the dimension of the kernel of $\capmat^{\gamma,\frac{\pi}{L}} - \tr(\capmat^{\gamma,\frac{\pi}{L}})/2$. After some algebraic manipulation, we obtain
    \begin{align*}
        \capmat^{\gamma,\frac{\pi}{L}} - \frac{1}{2}\tr(\capmat^{\gamma,\frac{\pi}{L}}) =
        \begin{pmatrix}
            \frac{\gamma  (s_2-s_1)}{2 s_1 s_2} &
   \frac{\gamma  \left(s_1-e^{\gamma }
   s_2\right)}{\left(e^{\gamma }-1\right) s_1
   s_2} \\
 \frac{\gamma  \left(e^{\gamma }
   s_1-s_2\right)}{\left(e^{\gamma }-1\right)
   s_1 s_2} & \frac{\gamma  (s_1-s_2)}{2
   s_1 s_2}
        \end{pmatrix},
    \end{align*}
    so that its eigenvectors read
    \begin{align*}
        v_{\pm} = \begin{pmatrix}
            \pm\sqrt{2} \sqrt{e^{\gamma } \left(\cosh (\gamma ) \left(s_1^2-6 s_1
   s_2+s_1^2\right)+(s_1+s_2)^2\right)}-e^{\gamma } s_1+\left(e^{\gamma }-1\right)
   s_2+s_1 \\
   2 e^{\gamma } s_1-2 s_2
        \end{pmatrix},
    \end{align*}
    which obviously agree for $\gamma$ as above.
\end{proof}

We denote by 
\begin{align}
\gamma_c(s_1,s_2) \label{eq: def gammac}
\end{align}
the unique critical $\gamma$ satisfying \eqref{eq: critical gamma}.

In \cref{fig: critical gamma} we consider a periodic chain of dimers having $s_1=1$ and $s_2=2$. We plot the two eigenvalues in the complex plane as $\alpha$ varies across the Brillouin zone, from $-\pi/L$ to $\pi/L$. \cref{lemma: critical gamma} tells us that, for these parameter values, the eigenvalues will coalesce when $\gamma_c = 0.73899$. The crossing of this critical $\gamma_c$ corresponds to a fundamental behavioural change in the band functions.
\begin{figure}
    \includegraphics[width=\textwidth]{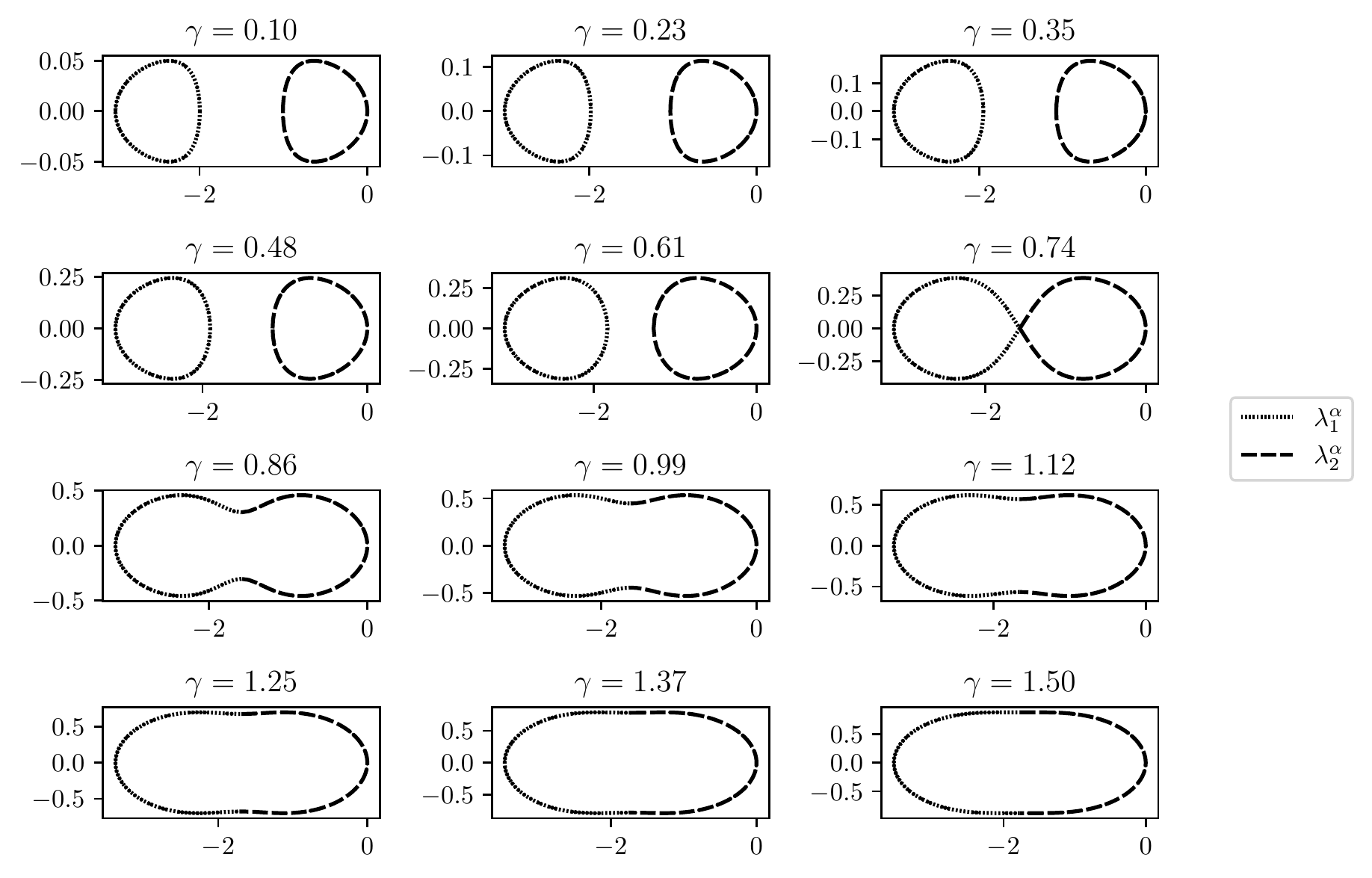}
    \caption{Eigenvalue behaviour in the complex plan as $\alpha$ varies over the Brillouin zone for different values of $\gamma$ for a system of periodically repeated dimers with $s_1=1$ and $s_2=2$. The value $\gamma_c(1,2)=0.73899$ corresponding to a system with exceptional point signs the change from open to close band structure.}
    \label{fig: critical gamma}
\end{figure}
This observation leads to the following results.
\begin{proposition}\label{prop: winding of eigenvalues}
    Consider a system of periodic dimers with spacings $s_1$ and $s_2$ and band functions $\omega_{i}^\alpha$ for $i=1,2$. There exists complex points $E_i\in\C$ with non-zero winding of $\omega_{i}^\alpha$ for $i=1,2$ if and only if $0<\gamma<\gamma_c(s_1,s_2)$.
\end{proposition}
\begin{proof}
    The main observation of the proof is that for $0<\gamma<\gamma_c(s_1,s_2)$ we have, for $\sigma=(12)$ the permutation of two symbols,
    \begin{align*}
        \Re(\omega^{\frac{\pi}{L}}_i)\neq\Re(\omega^{\frac{\pi}{L}}_{\sigma(i)}),\quad \Im(\omega^{\frac{\pi}{L}}_i)=\Im(\omega^{\frac{\pi}{L}}_{\sigma(i)})
    \end{align*}
    while for $\gamma>\gamma_c(s_1,s_2)$ the opposite happens
    \begin{align*}
        \Re(\omega^{\frac{\pi}{L}}_i)=\Re(\omega^{\frac{\pi}{L}}_{\sigma(i)}), \quad \Im(\omega^{\frac{\pi}{L}}_i)\neq\Im(\omega^{\frac{\pi}{L}}_{\sigma(i)}).
    \end{align*}
    This is easily seen from the fact that $\tr( \capmat^{\gamma,\frac{\pi}{L}})^2-4\det(\capmat^{\gamma,\frac{\pi}{L}})$ changes sign at $\gamma=\gamma_c(s_1,s_2)$ as we can see from the formula in the proof of \cref{lemma: critical gamma}. Combining this with \cref{lemma: capmat -alpha is conjugation} we can conclude that $\omega_i^\alpha$ draws a closed curve on the complex plane if and only if $0<\gamma<\gamma_c(s_1,s_2)$. The non-triviality of the winding then follows by the fact that the real and the imaginary parts of $\omega_i^\alpha$ are periodic with period $\frac{2\pi}{L}$.
\end{proof}
\begin{corollary}
    The \emph{vorticity}
    \begin{align} \label{defnu} 
    \nu_\gamma \coloneqq \frac{1}{2\pi}\int_{-\frac{\pi}{L}}^{\frac{\pi}{L}} \frac{\partial}{\partial \alpha}\arg(\omega_2^\alpha-\omega_1^\alpha)\dd \alpha
    \end{align}
    of a system of periodically repeated dimers is non-zero if and only if $0<\gamma<\gamma_c(s_1,s_2)$.
\end{corollary}

\begin{figure}
    \centering
    \begin{subfigure}[t]{0.45\textwidth}
        \centering
        \includegraphics[width=\textwidth]{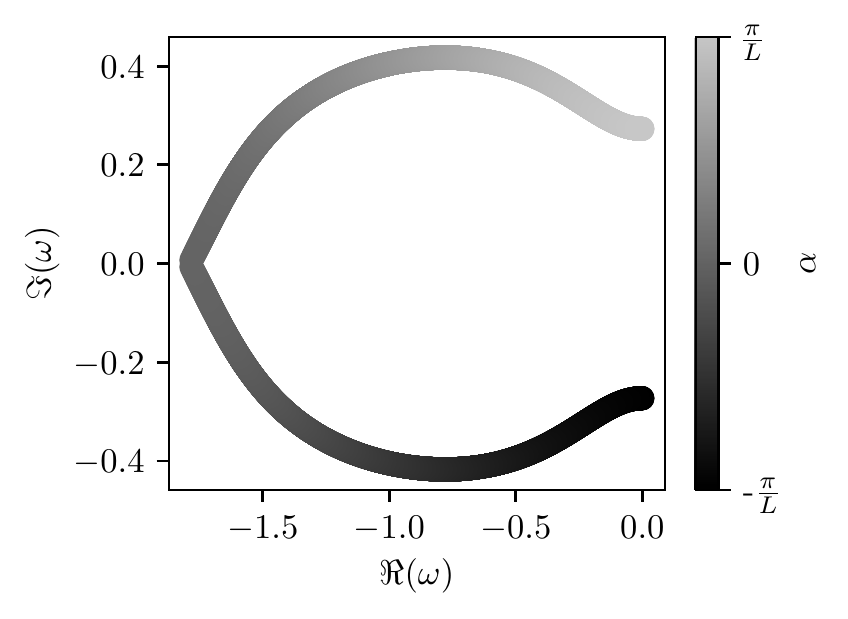}
        \caption{Band function without winding ($\gamma_c<\gamma=0.9$).}
        \label{fig: open band}
    \end{subfigure} \hfill
    \begin{subfigure}[t]{0.45\textwidth}
        \centering
        \includegraphics[width=\textwidth]{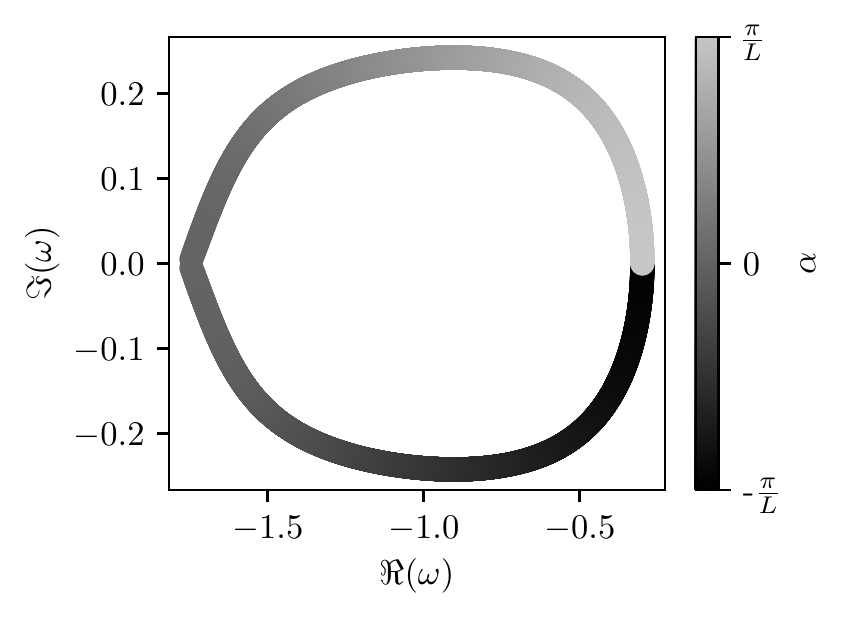}
        \caption{Band function with winding and non-zero vorticity ($\gamma=0.5<\gamma_c$).}
        \label{fig: non zero vorticity band}
    \end{subfigure}\\
    \begin{subfigure}[t]{0.45\textwidth}
        \centering
        \includegraphics[width=\textwidth]{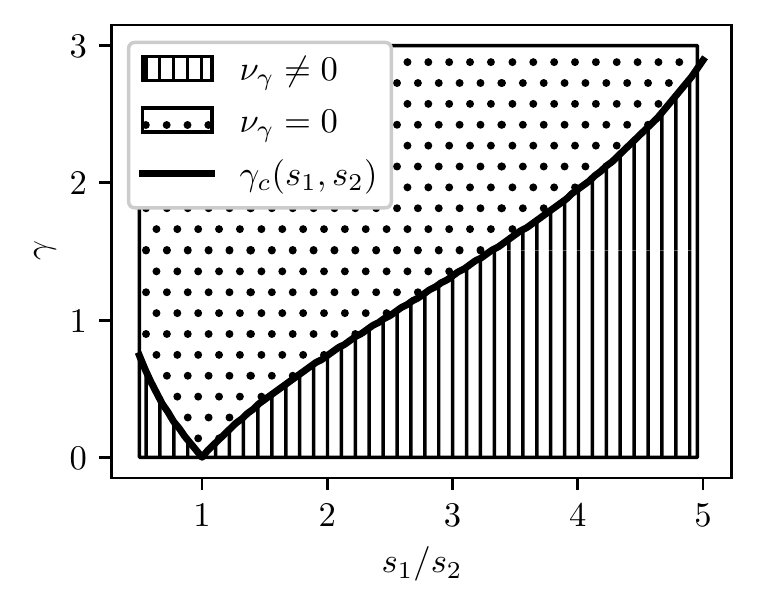}
        \caption{Relation between the dimer spacing $s_1$, $s_2$, $\gamma_c$ and vorticities.}
        \label{fig: gammac, nu and si}
    \end{subfigure}
    \caption{Band function winding and its relation to $\gamma_c$. The crossing of $\gamma_c$ induces a change in the winding of the band functions and thus of the vorticity $\nu_\gamma$. Computed for a system of $N=2$ periodically repeated resonators with spacing $s_1=1$ and $s_2=2$.}
    \label{fig: gammac and vorticity}
\end{figure}

In \cref{fig: gammac and vorticity}, we plot just one of the eigenvalues in the complex plane, to fully illuminate the behaviour we observed in \cref{fig: critical gamma}. We can clearly see the fundamental change that occurs at the transition when $\gamma$ crosses the value $\gamma_c(s_1,s_2)$. Whereas when $\gamma<\gamma_c$ the eigenvalue forms a single closed loop, when $\gamma>\gamma_c$ this loop is broken open to form a C-shaped curve that connects to the other eigenvalue to form a single loop. At the transition between these two states, when $\gamma=\gamma_c$, the two eigenvalues form two touching closed loops. In \cref{fig: gammac and vorticity}(C), we show how the critical value $\gamma_c$ varies as a function of the ratio of the spacings between the two resonators. Of course, when $s_1=s_2$ the system is just a single repeating resonator so this critical value vanishes. Otherwise, it is strictly positive. 

\subsection{Complex band theory and generalised Brillouin zone}\label{sec: complex bands}

We saw in \cref{sec: skin effect} that the eigenmodes of a finite system of subwavelength resonators with imaginary gauge potential have exponentially decaying amplitudes. Further, we saw that the system displays signs of an ``infinite'' order exceptional point in the limit as its size becomes large. Our final results of this paper concern the precise nature of the limiting spectrum, as the size becomes arbitrarily large. 

A crucial first question is how to understand the spectrum of the limiting operator. For periodic (Hermitian) systems, the natural tool to apply is the Floquet-Bloch transform. However, as shown in \cref{sec: periodic case1}, this does not lead to the correct limit of the set of the eigenvalues associated with the finite structure as its size goes to infinity. This is because the eigenmodes of the infinite periodic system are not Bloch modes, in the traditional sense. Instead, we will consider the Floquet-Bloch transformation with complex quasi-periodicity, so that we are able to account for the potential growth or decay of the amplitude of the eigenvectors \cite{borisov.fedotov2022Bloch}. 

Recently, the truncated Floquet-Bloch transformation has been applied to recover information about the band structure from finite approximations in Hermitian systems \cite{ammari.davies.ea2023Spectrala, ammari.davies.ea2023Convergence}. Here, since we are considering the quasi-periodicity to be complex, we apply a slightly different method. We consider a system of $N$ resonators. For a resonance $\omega_j$, we compute the corresponding quasi-periodicity by

\begin{align}
\alpha_j \coloneqq \argmin_{\alpha\in \C} \vert \omega_j - \upomega^\alpha \vert, \label{eq: approx alpha}
\end{align}
where $\upomega^\alpha$ is the  subwavelength eigenfrequency of a system with one resonator repeated periodically. In \cref{fig: complex band structre}, we display this procedure for an array of $N=60$ resonators. In \cref{fig: minimising region}, we show the minimising region of \eqref{eq: approx alpha} and see clear local minima of the function $\vert \omega_j - \upomega^\alpha \vert$ for values of $\alpha$ that are symmetric about the imaginary axis. In \cref{fig: complex alphas} we show the resulting complex quasi-frequencies $\alpha_j$, plotted against their indices, and in \cref{fig: complex band functions}, we plot the complex band functions $\alpha_j\mapsto \omega_j$. In \cref{sec: approx band}, we present the analogous numerical results for a system of periodically repeated dimers both in the Hermitian case and non-Hermitian case with complex material parameters.

\begin{figure}[h!]
    \centering
    \begin{subfigure}[t]{0.48\textwidth}
        \centering
        \includegraphics[width=\textwidth]{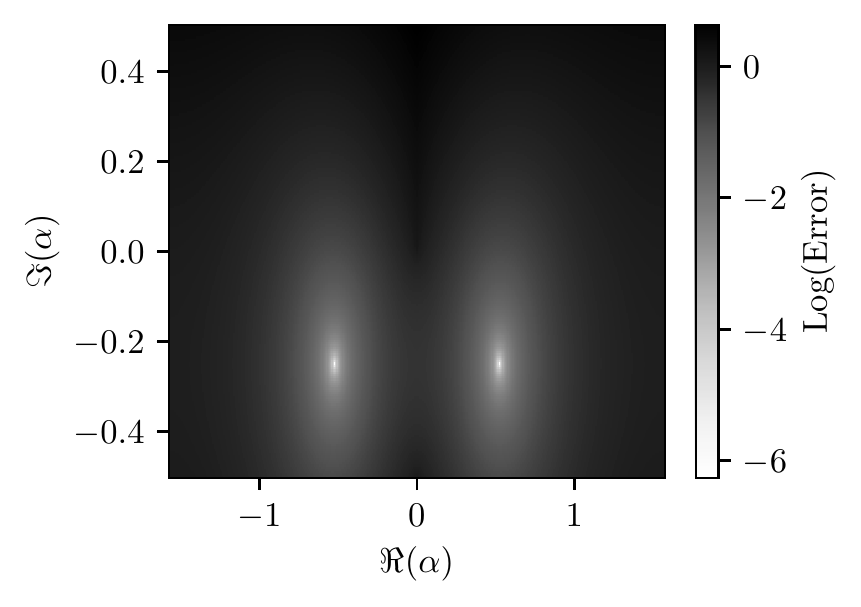}
        \caption{Minimising region for $\vert \omega_j - \upomega^\alpha \vert$ for $j=20$.}
        \label{fig: minimising region}
    \end{subfigure}
    \hfill
    \begin{subfigure}[t]{0.48\textwidth}
        \centering
        \includegraphics[width=\textwidth]{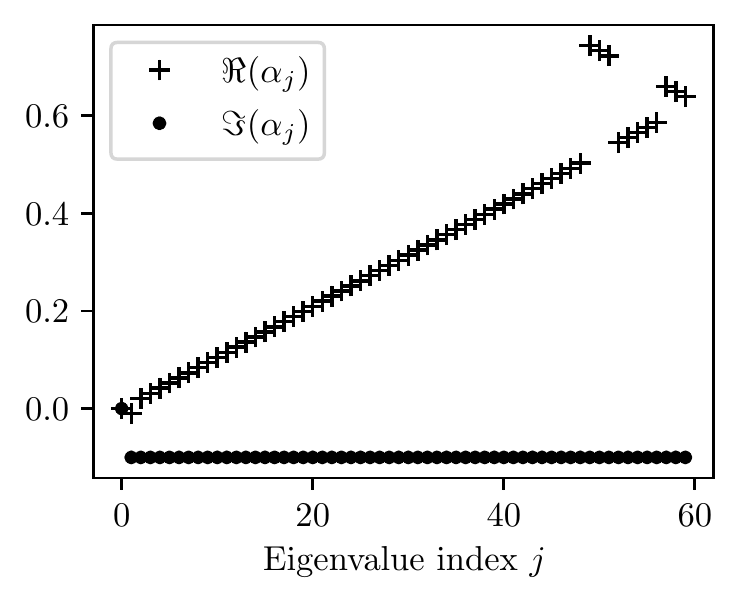}
        \caption{Computed $\alpha_j$ with non-trivial imaginary part.}
        \label{fig: complex alphas}
    \end{subfigure}\\
\begin{subfigure}[t]{0.48\textwidth}
        \centering
        \includegraphics[width=\textwidth]{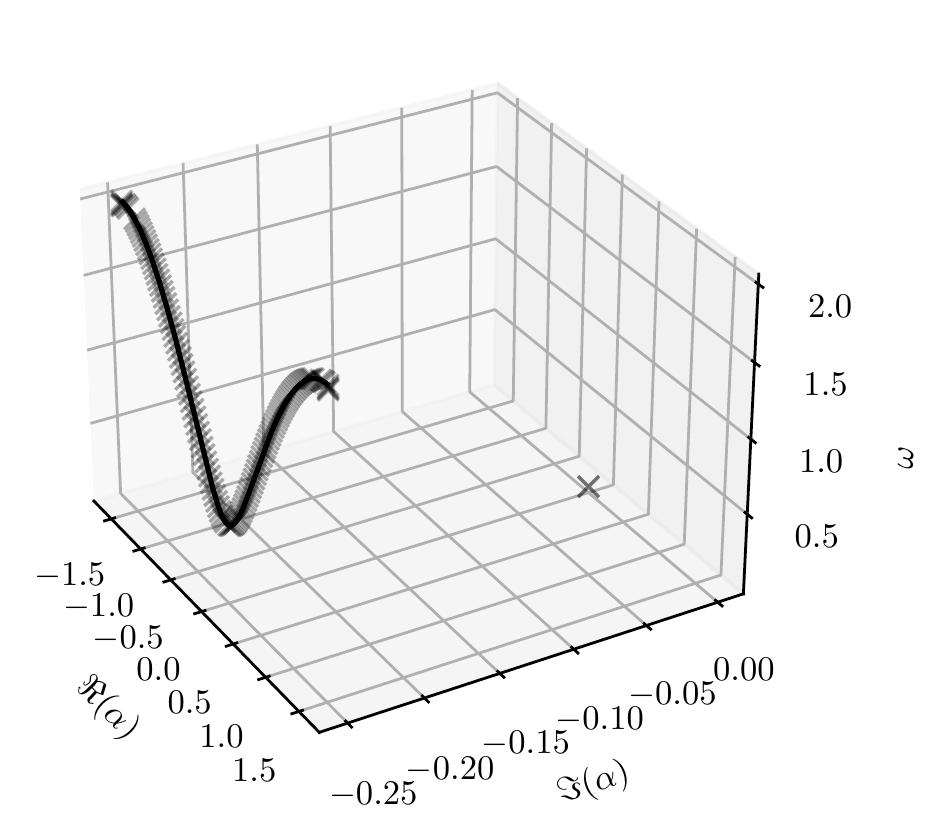}
        \caption{Continuous and discrete complex band functions in $\C\times\R$. The solid line shows the continuous complex band structure from \eqref{cvthm} item (i) while the gray crosses show the
pairs $(\alpha_j, \omega_j)$ from \eqref{eq: approx alpha} from the finite structure.}
        \label{fig: complex band functions}
\end{subfigure}
    \caption{Complex band structure of a system of $N=60$ resonators with $\ell=s=1$ and $\gamma=1$. The non-trivial imaginary part encodes the exponential decay of the eigenmodes presented in \cref{sec: skin effect}.}
    \label{fig: complex band structre}
\end{figure}

The above numerical results suggest to introduce the following set, called the \emph{generalised Brillouin zone}, 
\begin{equation} \label{gbz}
\mathcal{Y}^*:= \big\{ (\alpha, \beta(\alpha)) \in Y^* \times \R: \lambda_i^{\alpha + \i \beta(\alpha)}  \in \R^+ \mbox{ for }  i=1 \mbox{ or } 2 \big\}. 
\end{equation}
Here, $\lambda_i^{\alpha +\i \beta(\alpha)}$ for $i=1,2$, are defined in Lemma \ref{lemmaN=2}. The value of this set is immediately clear from the following convergence result. 
\begin{theorem} \label{cvthm}
\begin{enumerate}[(i)]
        \item Consider a chain of $N$ resonators with $s_i=s$ and $\ell_i=\ell$ for all $1\leq i\leq N$. Let the subwavelength eigenfrequency $\lambda_k$ be defined by (\ref{eq: eigenvalues capmat}). Then, there exists $(\alpha, \beta(\alpha)) \in \mathcal{Y}^*$ such that $|\lambda_k - \lambda^{\alpha+ \i \beta(\alpha)}| \rightarrow 0$ as  $N\rightarrow +\infty$;
\item Consider  a system of dimers, i.e. with $\ell_i=\ell=1$. Then, the same result as in (i) holds as the number of dimers goes to infinity.  
\end{enumerate}
\end{theorem}
\begin{proof}
    One first notices that because of the boundedness of the spectrum and after a possible renaming, $\lambda_k$ have to converge in $\R$ as $N\to+\infty$. The key idea of the proof is to use the
explicit formulations from \cref{lemmaN=2} and Picard's little theorem. Indeed, this is enough for item (i) as $\C\ni\alpha\mapsto \lambda^\alpha\in \C$ is entire and so in particular
$\mathcal{Y}^*\ni\alpha\mapsto \lambda^\alpha$ is surjective to $\R$ up to possibly one point. For item (ii) the application of Picard's little theorem to the argument of the square root shows that
$\C\ni\alpha\mapsto \lambda_i^\alpha\in \C$ for $i=1,2$ fails to be surjective because of the branch cut, but the opposite signs in $\lambda_1^\alpha$ and $\lambda_2^\alpha$ compensate, concluding the
proof.
\end{proof}

In \cref{fig: complex band functions}, we plot both the discrete pairs $(\alpha_j,\omega_j)$ and the continuous curve, computed over $\mathcal{Y}^*$. We see perfect agreement between the discrete and continuous data, to the extent that it is difficult to distinguish the continuous curve from the discrete crosses. This agreement is remarkable given that the finite structure here is still relatively small, having $N=60$ resonators. Further numerical results, including the extension to the case of resonator dimers, can be found in \cref{sec: approx band}.


\section{Concluding remarks}
We have derived from first principles the mathematical theory of the non-Hermitian skin effect arising in subwavelength physics in one dimension. Through a gauge capacitance matrix formulation, we obtained explicit asymptotic expressions for the subwavelength eigenfrequencies and eigenmodes of the structure. This allowed us to characterise the system's fundamental behaviours and reveal the mechanisms behind them. In particular, the exponential decay of eigenmodes (the skin effect) for a system of finitely many resonators was shown to be induced by the Fredholm index of an associated Toeplitz operator. We also showed how the system behaves as its size becomes large and developed appropriate tools to understand this spectral convergence.

A natural question to ask is how the non-Hermitian system considered here compares to other non-Hermitian systems. For example, while the non-Hermiticity here arises due to the imaginary gauge potential, it is possible to break Hermiticity by simply making the material parameters complex valued. This has been studied previously in the subwavelength setting \cite{ammari.davies.ea2022Exceptional, ammari.barandun.ea2023Edge}. It has been shown that, under the assumption of PT-symmetry, these non-Hermitian systems support exceptional points \cite{ammari.davies.ea2022Exceptional} and that the analysis of infinite periodic structures can be restricted to the standard (real) Brillouin zone  \cite{ammari.barandun.ea2023Edge}. As a result, the spectrum of the finite structure converges to the band structure of the infinite one. In \cref{sec: PT}, we show that this is due to the fact that, away from exceptional points, these non-Hermitian systems with complex material parameters can be mapped to a Hermitian system under an appropriate transformation. That is, they are equivalent to a Hermitian one away from exceptional points. Conversely, the system considered in this work, which is non-Hermitian due to the introduction of an imaginary gauge potential, is fundamentally distinct from the Hermitian system and a generalised (or complex) Brillouin zone must be considered to understand the limiting spectrum. 

The explicit theory we have developed is only possible because of the simple structure of the gauge capacitance matrix and the rich literature on (tridiagonal) Toeplitz matrices and perturbations thereof. The theory of systems with periodically repeated cells of $K$ resonators remains incomplete as no similar result to \cref{lemma: eigenpairs tilde T n} is known for block-Toeplitz matrices. However, we have shown numerically that the phenomena generalise to systems of many resonators (in Appendices~\ref{sec: dimers} and~\ref{sec: approx band}). The other important generalisation of this work is to higher-dimensional systems, in which it is well known that the skin effect can be realised \cite{skinadd6}.
This will be the subject of a forthcoming publication \cite{3DSkin}.

\addtocontents{toc}{\protect\setcounter{tocdepth}{0}}
\section*{Acknowledgments}
The work of JC was supported by Swiss National Science Foundation grant number 200021–200307. The work of BD was supported by a fellowship funded by the Engineering and Physical Sciences
Research Council under grant number EP/X027422/1. 
\bigskip

\section*{Code availability}
The data that support the findings of this work are openly available at \\ \href{https://doi.org/10.5281/zenodo.8081076}{https://doi.org/10.5281/zenodo.8081076}.

\appendix
\addtocontents{toc}{\protect\setcounter{tocdepth}{1}}
\section{Toeplitz theory}\label{sec: teoplitz theory}
Item (iii) of \cref{cor: properties of cap mat} is a very powerful result as it allows us to apply the rich theory of Toeplitz matrices. In this section, we will restate some useful results adapted to our needs.

\begin{definition}\label{def: Teoplitz stuff}
    An $N\times N$ \emph{Toeplitz matrix} $T$ is a matrix whose entries are constant on the diagonals, in the sense that there exits $a_{1-N},\dots a_0,\dots,a_{N-1}\in \C$ so that 
    \begin{align}
    \label{eq: definition of toeplitz matrix}
    T_{i,j} = a_{i-j},
    \end{align}
    for $i,i=1,\dots,N$. A semi-infinite matrix $(T_{i,j})_{i,j\in\N}$ respecting \eqref{eq: definition of toeplitz matrix} for a sequence $(a_i)_{i\in\Z}\subset \C$ is called a \emph{Toeplitz operator}.
    The \emph{symbol} of a Toeplitz operator is the function
    \begin{align*}
    f_T:S^1&\to \C\\
    z&\mapsto \sum_{k\in\Z}a_k z^k,
    \end{align*}
    where $S^1$ is the unit circle.
    \end{definition}

    For an operator $A$ we denote by
    \begin{align*}
        \sp(A)\coloneqq\left\{ \lambda\in\C:A-\lambda \text{ is not invertible}\right\}
    \end{align*}
    the spectrum of $A$. For some $\epsilon>0$, we denote by $\sp_\epsilon(A)$ the pseudo-spectrum of $A$ defined equivalently \cite[Section 4]{trefethen.embree2005Spectra} by
    \begin{align*}
    \lambda\in\sp_\epsilon(A) &\Leftrightarrow \Vert (A -\lambda)\inv \Vert>\epsilon\inv\\
    &\Leftrightarrow \lambda\in \sp(A+E)\text{ for some } E \text{ such that } \Vert E \Vert < \epsilon\\
    &\Leftrightarrow \lambda\in \sp(A) \text{ or } \Vert (A-\lambda)u\Vert < \epsilon \text{ for some } u \text{ such that } \Vert u \Vert = 1.
    \end{align*}
    
    The spectrum of a Toeplitz operator is well understood, as the next lemma shows \cite[Theorem 1.10]{bottcher.silbermann1999Introduction}.
    \begin{lemma}\label{lemma: specturm of toeplitz operator}
    Let $T$ be a Toeplitz operator with symbol $f_T\in L^\infty$. Then 
    \begin{align*}
    \sp(T)=\underbrace{\left\{\lambda\in\C:T-\lambda \text{ is not Fredholm}\right\}}_{\eqqcolon\sp_{\mathrm{ess}}(T)} \cup \left\{\lambda\in\C:\Ind(T-\lambda)\neq 0\right\},
    \end{align*}
    where $\Ind(A)$ is the Fredholm index of $A$.
    In particular, for $f_T$ continuous we have
    \begin{align*}
        \sp(T)= f_T(S^1)\cup \left\{\lambda\in\C:w(f_T,\lambda)\neq 0\right\},
    \end{align*}
    where  $w(f,\lambda)$ is the winding number of $f$ around $\lambda$.
    \end{lemma}
    
    From \cref{lemma: specturm of toeplitz operator} we can see that the spectrum of a Toeplitz operator potentially has non-zero measure in the complex plane. It is therefore impossible that the spectrum of a sequence of Toeplitz matrices converges to the one of a Toeplitz operator in general. The following result  \cite[Theorem 7.3]{trefethen.embree2005Spectra} shows, however, that such a convergence result exists if we consider pseudospectra.
    
    \begin{lemma}\label{lemma: convergence of speudospectrum topelitz op}
        Let $T$ be a Toeplitz operator with continuous symbol $f_T$ and let $T_N$ be the truncation of $T_N$ to the upper-left $N\times N$ submatrix. Then for any $\epsilon>0$
        \begin{align*}
        \lim_{N\to\infty} \sp_\epsilon(T_N) = \sp_\epsilon(A),
        \end{align*}
        in the Hausdorff metric, meaning 
        \begin{align*}
            \sp_\epsilon(A) = \{z\in \C: \exists z_N \in\sp_\epsilon(T_N): z_N\to z \}.
        \end{align*}
    \end{lemma}
    
    While \cref{lemma: convergence of speudospectrum topelitz op} describes the convergence of the pseudospectra of $T_N$, we also know how the corresponding eigenvectors behave. For example, we have the following lemma, from \cite[Section 7]{trefethen.embree2005Spectra}.
        
    \begin{lemma}\label{lemma: facts on eigenvectors teopliz}
        Let $T$ be a Toeplitz operator with continuous symbol $f_T$ and let $T_N$ be the truncation of $T$ to the upper-left $N\times N$ submatrix. Let $\lambda\in\C$  be such that $w(f_T,\lambda)<0$. Then, if $f_T$ is sufficiently smooth, the corresponding eigenvector decays exponentially. Furthermore, there exists some $M>1$ which is such that $\lambda$ is a $M^{-N}$ pseudoeigenvalue of $T_N$ with corresponding eigenvector $v_N$ satisfying
        \begin{align*}
        \frac{\vert v_N^{(j)}\vert}{\max_j \vert v_N^{(j)}\vert}\leq M^{-j},
        \end{align*}
        where $v_N^{(j)}$ is the $j$-th component of $v$. 
    \end{lemma}
    
    \addtocontents{toc}{\protect\setcounter{tocdepth}{1}}
    \subsection*{Tridiagonal Toeplitz matrices}\label{sec: tridig teoplitz theory}
    As the capacitance matrices we handle are tridiagonal, we focus now on the well understood theory of tridiagonal Toeplitz matrices. We will denote $N\times N$ tridiagonal Toeplitz matrices by their diagonal entries
    \begin{align*}
    T_N(a,b,c) \coloneqq \begin{pmatrix}
    b & c & & & &\\
    a & b & c &&&\\
    & \ddots & \ddots & \ddots &&\\
    & & a & b & c\\
    &&& a & b
    \end{pmatrix}.
    \end{align*}
    The eigendecomposition of tridiagonal Toeplitz matrices is well understood, as the following result \cite[Section 2]{noschese.pasquini.ea2013Tridiagonal} shows. 
    \begin{lemma}[Eigenvalues and eigenvectors of tridiagonal Toepliz matrices]\label{lemma: eigenpairs of tridiag toeplitz}
        The eigenvalues of $T_N(a,b,c)$ are given by
        \begin{align*}
            \lambda_k = b + 2\sqrt{ac}\cos\left(\frac{k}{N+1}\pi\right),\quad 1\leq k\leq N,
        \end{align*}
        and the corresponding eigenvectors
        \begin{align*}
        u_k^{(i)} = \left(\frac{a}{c}\right)^{\frac{i}{2}}\sin\left(\frac{ik}{N+1}\pi\right), \quad 1\leq k\leq n \text{ and } 1\leq i\leq N,
        \end{align*}
        where $u_k^{(i)}$ is the $i$-th coefficient of the $k$-th eigenvector.
    \end{lemma}
    The spectral theory of tridiagonal matrices that are almost Toeplitz has been object of many studies. In this work we are interested in $N\times N$ matrices with the modified Toeplitz structure 
    \begin{align*}
        \tilde{T}_N(a,b,c) \coloneqq \begin{pmatrix}
        b+a & c & & & &\\
        a & b & c &&&\\
        & \ddots & \ddots & \ddots &&\\
        & & a & b & c\\
        &&& a & b+c
        \end{pmatrix}.
        \end{align*}
    For matrices of the type $\tilde{T}_N(a,b,c)$ the spectral theory is fully understood as well, as the following lemma \cite[Theorem 3.1]{yueh.cheng2008Explicit} shows.
    \begin{lemma}\label{lemma: eigenpairs tilde T n}
        Suppose that $ac\neq 0$, $a+b+c=0$ and let $\mu$ be an eigenvalue of $\tilde{T}_N(a,b,c)$ then either $\mu= \mu_1:=0$ and the corresponding eigenvector is $v_1=\bm 1$ or
        \begin{align*}
        \mu= \mu_k := b+2\sqrt{ac}\cos\left(\frac{\pi}{N} (k-1)\right), \quad 2\leq k\leq N,
        \end{align*}
        and the corresponding eigenvector 
        \begin{align*}
            v_k^{(j)} = \left(\frac{a}{c}\right)^{\frac{j-1}{2}}\left(a\sin\left(\frac{i (k-1) \pi}{N}\right) - a\sqrt{\frac{a}{c}}\sin\left(\frac{(j-1) (k-1) \pi}{N}\right)\right), 
            \end{align*}
            for $2\leq k\leq N \text{ and } 1\leq j\leq N.$
    \end{lemma}
    It should be noted that the condition $a+b+c=0$ guarantees that $\mu_1=0$.
    \begin{proposition}
        Let $N\in \N$ and $(\lambda_{i,N},u_{i,N})$ for $1\leq i\leq N$ and $(\mu_{i,N},v_{i,N})$ for $2\leq i \leq N$ be the eigenpairs of $T_N(a,b,c)$ and $\tilde{T}_N(a,b,c)$ respectively so that the conditions of \cref{lemma: eigenpairs tilde T n} are satisfied. Then there exists an injective map $\sigma_N:\{2,\dots,N\}\to\{1,\dots,N\}$ such that for any $i=2,\dots,N,$
        \begin{align*}
            \abs{\mu_{\sigma_N(i),N}-\lambda_{i,N}}\to 0\quad \text{as}\quad N\to \infty
        \end{align*}
        and 
        \begin{align*}
            \Vert v_{\sigma_N(i),N}-u_{i,N}\Vert \to 0\quad \text{as}\quad N\to \infty.
        \end{align*}
    \end{proposition}
    \begin{proof}
        This is now trivial using the explicit formulations from \cref{lemma: eigenpairs of tridiag toeplitz,lemma: eigenpairs tilde T n}.
    \end{proof}
    Despite not having a surjective mapping, there is only one eigenvalue of the perturbed matrix that cannot be approximated via the Toeplitz matrix and this eigenpair is known explicitly ($\mu_1=0$, $v_1=\bm 1$).

\section{Dimer systems of subwavelength resonators with imaginary gauge potentials }\label{sec: dimers}
In this appendix, we briefly present some numerical illustrations of the skin effect in a system of finitely many dimers. \Cref{fig:skin modes dimers} shows the eigenvectors in a system of subwavelength resonator dimers with imaginary gauge potential, which are clearly strongly localised at one edge. This is a skin effect analogue to the simpler case studied in \cref{sec: skin effect}, with the corresponding modes plotted in \cref{fig:skin modes}. Then, \Cref{fig:skin pseudospectrum dimers} is the analogue of \cref{fig: winding and pseudospecturm}, which shows that the spectrum of repeated dimers presents a bipartition and remains purely real.

\begin{figure}[htb]
    \centering
    \begin{subfigure}[t]{0.48\textwidth}
    \includegraphics[width=\textwidth]{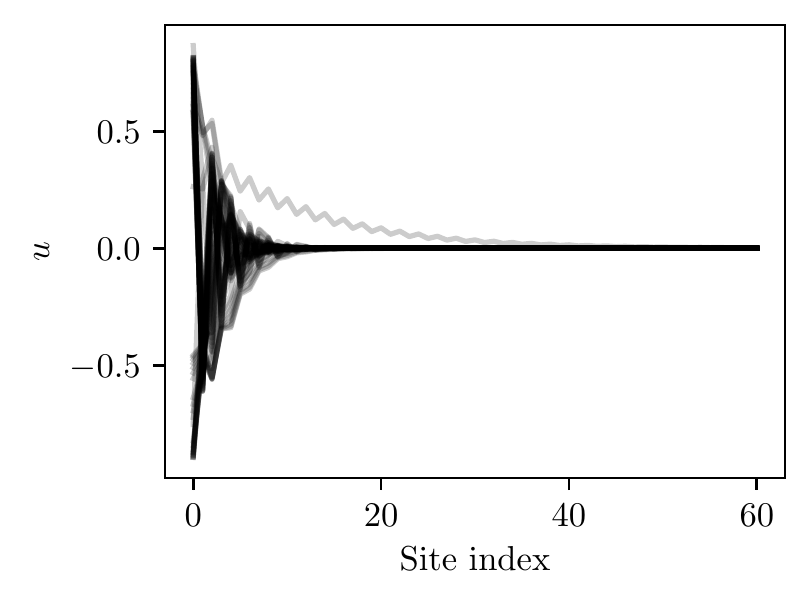}
    \caption{Eigenvector localisation showing skin effect. All modes are exponentially decaying and compensated on the left-hand side because of $\gamma>0$.}
    \label{fig:skin modes dimers}
    \end{subfigure}
    \hfill
    \begin{subfigure}[t]{0.48\textwidth}
        \includegraphics[width=\textwidth]{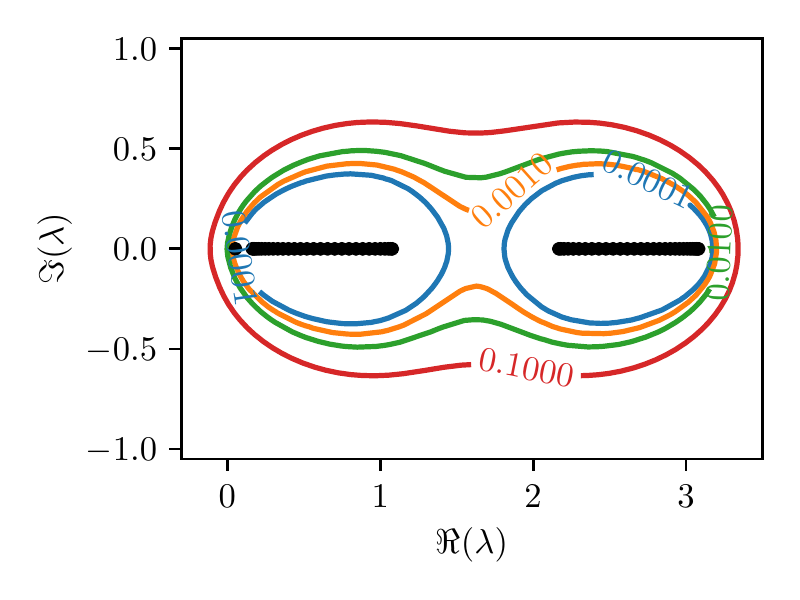}
        \caption{Eigenvalues $\lambda_i$ (bullets, purely real) and $\epsilon$-pseudospectra (boundary lines) for $\epsilon=10^{-k}$ for $k=0,\dots,4$.}
        \label{fig:skin pseudospectrum dimers}
    \end{subfigure}
    \caption{Eigenmodes and spectrum for a system of $N=61$ dimers with $\ell=1$, $s_1=1$ and $s_2=2$ and $\gamma=1$.}
\end{figure}

\section{Generalised discrete Brillouin zone for dimer systems}\label{sec: approx band}

In \cref{sec: complex bands}, we have shown that in order to describe an infinite array of subwavelength resonators with imaginary gauge potential, the correct quasiperiodicity space to use lies in a region of the complex plane. This generalised Brillouin zone is needed in order to take into account decaying eigenmodes. In this appendix, we first show that this is indeed a generalisation of the standard (real) Brillouin zone. In \cref{fig: computed alphas herm 1 res,fig: Computed bands hermitian 1 res} we consider a large system of equally spaced identical Hermitian resonators ($\gamma=0$). \cref{fig: computed alphas herm 1 res} shows the real and imaginary part of quasiperiodicities computed as described in \cref{sec: complex bands}. We notice that despite the optimisation region is a region of the complex plane, the resulting quasifrequencies are real as expected. \cref{fig: Computed bands hermitian 1 res} compares the computed discrete band structure and the actual band structure for the periodic case. There is very close agreement between them.

Then we compute the same for a large system of Hermitian ($\gamma=0$ and real material parameters) dimers. In this case, due to the presence of multiple bands, we slightly modify the procedure of \cref{sec: complex bands}. In this case we consider the following minimisation:
\begin{align*}
    \alpha_j \coloneqq \argmin_{\alpha\in \C} \vert \omega_j - \upomega^\alpha \vert,
\end{align*}
where $\upomega^\alpha$ is one of the two eigenvalues of the corresponding quasiperiodic capacitance matrix. We notice again that, as expected, the quasiperiodicities are real despite an optimisation occurring over $\C$ (\cref{fig: computed alphas herm 2 res}) and that the band functions are recovered accurately (\cref{fig: Computed bands hermitian 2 res}). 

To exemplify the approach, we take a system of dimers with complex material parameters $\kappa_b$ and $\rho_b$ (and $\gamma=0$). In this case, we optimise again with respect to one of the two eigenvalues of the corresponding quasiperiodic capacitance matrix. Also, as shown by \cref{fig: computed alphas non herm 2 res,fig: Computed bands non hermitian 2 res}, the quasiperiodicities are real despite our optimisation procedure is occurring over $\C$ and the band functions are recovered accurately. 


\begin{figure}
    \centering
    \begin{subfigure}[t]{0.48\textwidth}
        \centering
        \includegraphics[width=\textwidth]{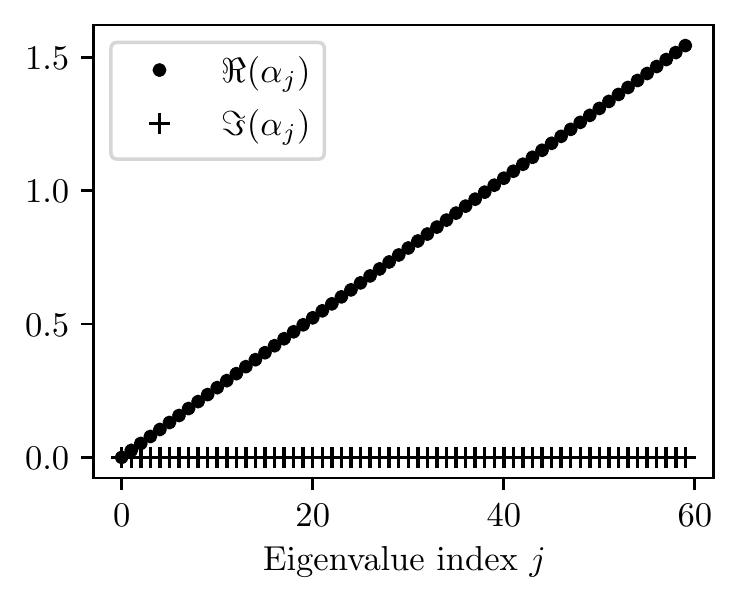}
        \caption{}
        \label{fig: computed alphas herm 1 res}
    \end{subfigure}
    \hfill
    \begin{subfigure}[t]{0.48\textwidth}
        \centering
        \includegraphics[width=\textwidth]{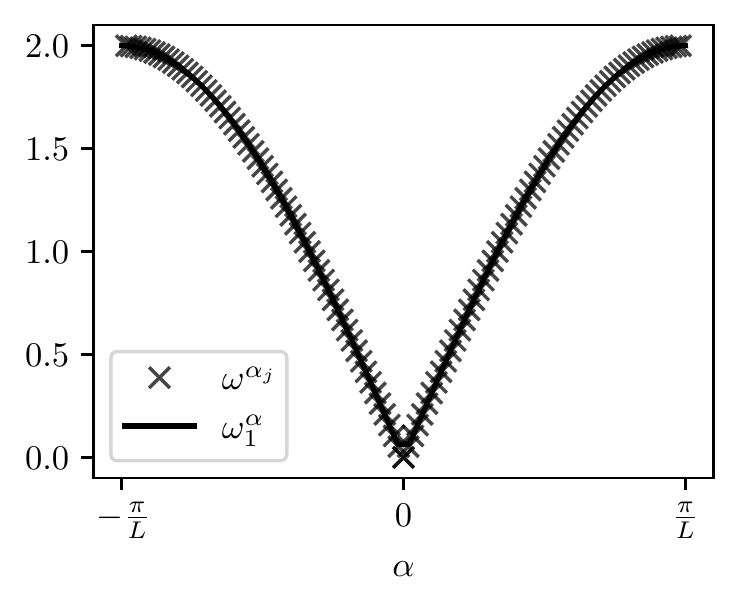}
        \caption{
        }
        \label{fig: Computed bands hermitian 1 res}
    \end{subfigure}\\

    \begin{subfigure}[t]{0.48\textwidth}
        \centering
        \includegraphics[width=\textwidth]{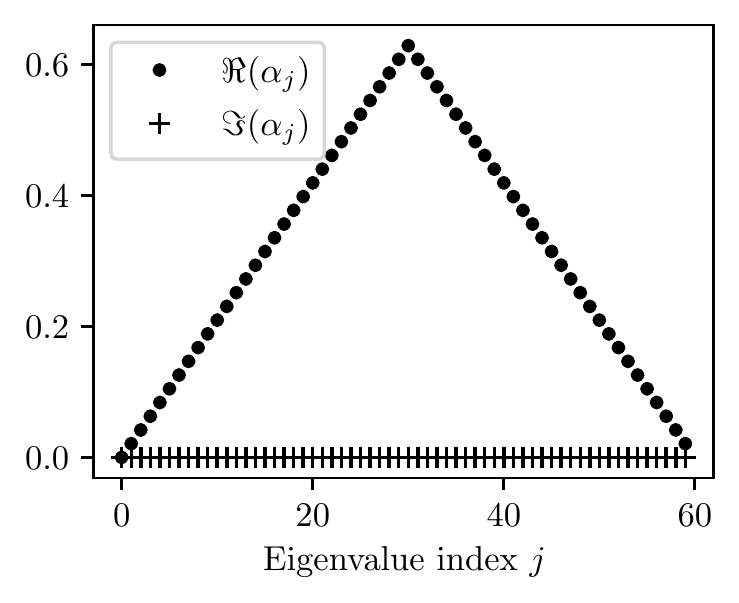}
        \caption{
        }
        \label{fig: computed alphas herm 2 res}
    \end{subfigure}
    \hfill
    \begin{subfigure}[t]{0.48\textwidth}
        \centering
        \includegraphics[width=\textwidth]{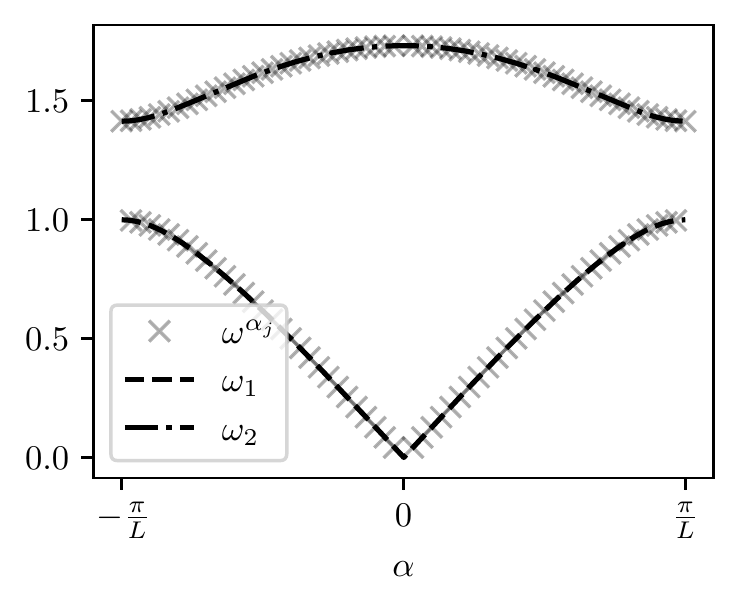}
        \caption{
        }
        \label{fig: Computed bands hermitian 2 res}
    \end{subfigure}\\
    \begin{subfigure}[t]{0.48\textwidth}
        \centering
        \includegraphics[width=\textwidth]{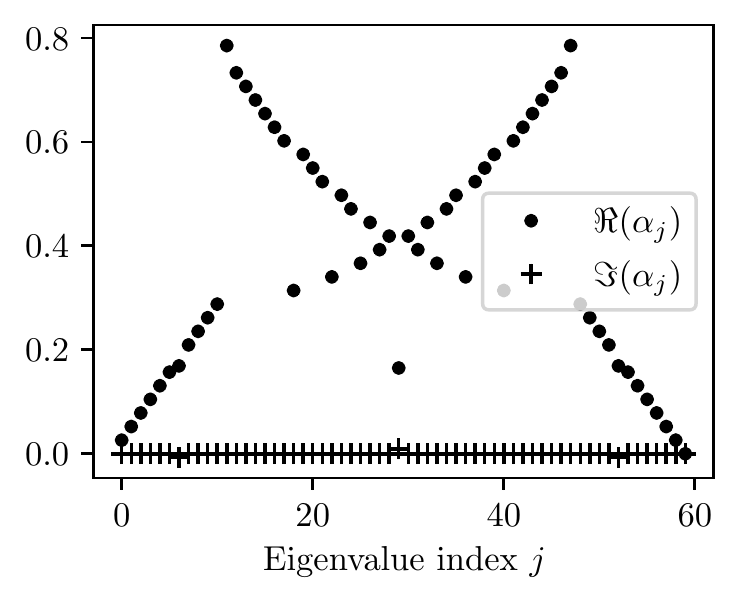}
        \caption{
        }
        \label{fig: computed alphas non herm 2 res}
    \end{subfigure}
    \hfill
    \begin{subfigure}[t]{0.48\textwidth}
        \centering
        \includegraphics[width=\textwidth]{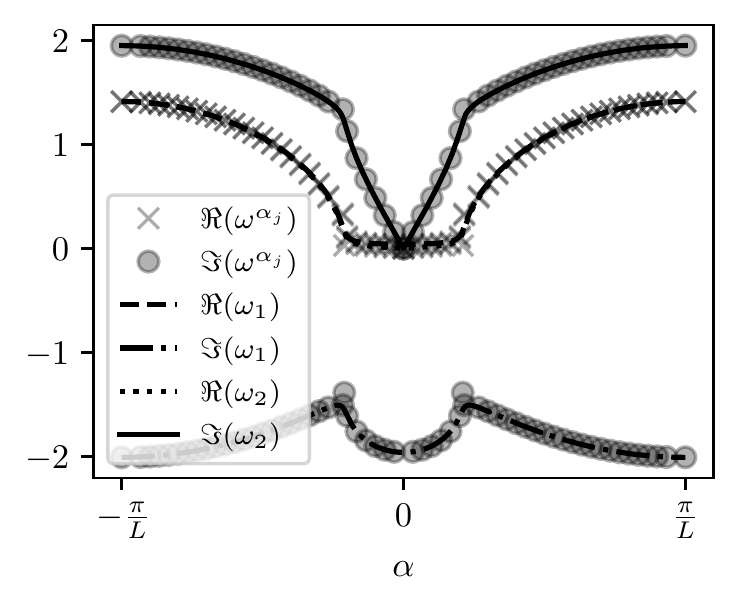}
        \caption{
        }
        \label{fig: Computed bands non hermitian 2 res}
    \end{subfigure}
    \caption{Discrete band approximation for various systems. \cref{fig: computed alphas herm 1 res,fig: Computed bands hermitian 1 res} show a system of $N=60$ Hermitian equally spaced identical resonators with $\ell=s=1$ and $\gamma=0$. \cref{fig: computed alphas herm 2 res,fig: Computed bands hermitian 2 res} show a system of $N=60$ Hermitian dimers with $\ell=1$, $s_1=1$, $s_2=2$ and $\gamma=0$. \cref{fig: computed alphas non herm 2 res,fig: Computed bands non hermitian 2 res} show a system of $N=60$ non-Hermitian dimers with complex material parameters $v_1=1+1.38\i$, $v_2=1-1.42\i$, $\ell=1$, $s=1$ and $\gamma=0$.}
    \label{fig: 6 figs generalised Brillouin zone}
\end{figure}

\section{Systems with complex material parameters and imaginary gauge transformations}
\label{sec: PT}

In this appendix, we highlight the fundamental difference between the current work and previous work on subwavelength non-Hermitian systems \cite{ammari.hiltunen2020Edge}. We show that systems with non-Hermiticity arising from complex material parameters are indeed, how \cref{fig: 6 figs generalised Brillouin zone} hits to, reducible to Hermitian systems away from
exceptional points. This implies that the  vorticity $\nu$ defined by the same formula as in (\ref{defnu})  is trivial, see \cite{ammari.hiltunen2020Edge}.

We consider the following ordinary differential equation (ODE):
\begin{align} \label{ODE}
    \psi\prii + VC \psi = 0,
\end{align}
where $V$ is a diagonal matrix encoding the (complex) material parameters and $C$ the (Hermitian) capacitance matrix. We assume that $VC$ is diagonalisable and invertible, i.e., we are away from
exceptional points. Equation (\ref{ODE}) describes the set of subwavelength eigenfrequencies. After a change of basis we can thus assume $VC = D = \diag(\lambda_1,\lambda_2)$. Introduce the transformation
\begin{align*}
G: \R &\to \R^2\\
t&\mapsto (\lambda_1^{-\frac{1}{2}}t,\lambda_2^{-\frac{1}{2}}t)
\end{align*}
and then consider
\begin{align*}
\phi(t) = \psi \circ G(t),
\end{align*}
which satisfies
\begin{align*}
    \phi(t)\prii = D\inv\cdot(\psi\prii \circ G(t)).
\end{align*}
Therefore, $\phi(t)$ satisfies the Hermitian ODE
\begin{align*}
    \phi\prii + \phi = 0
\end{align*}
as
\begin{align*}
    D(\phi\prii(t) + \phi(t)) = D D\inv(\psi\prii \circ G(t)) + D (\psi \circ G(t)) = (\psi\prii + D\psi)\circ G(t) = 0.
\end{align*}

\printbibliography

\end{document}
